\documentclass[12pt]{amsart}
\usepackage[T1]{fontenc}
\usepackage{lmodern,cite}
\usepackage{amsmath,amssymb,mathtools}
\usepackage[margin=1in]{geometry}
\usepackage{microtype}
\usepackage{needspace}
\usepackage{xcolor}
\usepackage[colorlinks=true,allcolors=blue]{hyperref}
\hypersetup{
  pdftitle={Positive sectional curvature on all smooth 7-spheres},
  pdfauthor={Yang-Hui He, Ziran Liu, Shing-Tung Yau},
  pdfsubject={Positive sectional curvature on smooth homotopy seven-spheres},
  pdfkeywords={Homotopy sphere, exotic seven-sphere, positive sectional curvature,
    Riemannian submersion, gluing}
}
\usepackage{amsthm}

\numberwithin{equation}{section}
\allowdisplaybreaks[2]
\newtheorem{theorem}{Theorem}[section]
\newtheorem{proposition}[theorem]{Proposition}
\newtheorem{lemma}[theorem]{Lemma}

\theoremstyle{remark}
\newtheorem{remark}[theorem]{Remark}
\newcommand{\Hh}{\mathbb H}
\newcommand{\Rr}{\mathbb R}
\newcommand{\eps}{\varepsilon}

\newcommand{\Hess}{\operatorname{Hess}}
\newcommand{\Gr}{\operatorname{Gr}}
\newcommand{\Id}{\operatorname{Id}}
\newcommand{\Ksec}{\operatorname{sec}}
\newcommand{\cK}{\mathcal K}
\newcommand{\hor}{\mathrm H}
\newcommand{\ver}{\mathrm V}
\title[Positive curvature on all smooth 7-spheres]{Positive sectional curvature on all smooth 7-spheres}

\author[Y.-H. He]{Yang-Hui He}
\author[Z. Liu]{Ziran Liu}
\author[S.-T. Yau]{Shing-Tung Yau}

\address[Y.-H. He]
{
London Institute for Mathematical Sciences, Royal Institution, London, W1S 4BS, UK}

\address[Y.-H. He]
{
Merton College, University of Oxford, OX1 4JD, UK}
\email{yh@lims.ac.uk}

\address[Z. Liu]
{Shanghai Institute for Mathematics and Interdisciplinary Sciences (SIMIS), Shanghai, 200433, China}
\address[Z. Liu]
{Research Institute of Intelligent Complex Systems, Fudan University, Shanghai, 200433, China}
\email{zliu@simis.cn}

\address[S.-T. Yau]
{Yau Mathematical Sciences Center, Tsinghua University, Beijing, 100084, China}
\email{styau@mail.tsinghua.edu.cn}

\date{September 23, 2026}
\subjclass[2020]{53C20, 57R55, 53C21}
\keywords{Homotopy sphere, exotic seven-sphere, positive sectional curvature,
Riemannian submersion, gluing}
\begin{document}
\begin{abstract}
We construct a smooth Riemannian metric with strictly positive sectional curvature on every smooth homotopy seven-sphere. For each smooth structure, we construct positively curved metrics on two seven-dimensional disks whose induced boundary metrics agree under the prescribed attaching map. Under this identification, a gluing theorem then yields the required smooth metric on the closed manifold. The construction applies to all 28 oriented diffeomorphism classes.
\end{abstract}
\maketitle

\section{Introduction}

Determining which smooth manifolds admit metrics of positive sectional
curvature is a fundamental problem in Riemannian geometry. On manifolds
homeomorphic to spheres, this question brings the distinction between
topology and differentiable structure into sharp focus. In 1956,
Milnor constructed smooth seven-dimensional manifolds homeomorphic, but
not diffeomorphic, to the standard sphere \(S^7\)~\cite{Milnor}.

A smooth homotopy seven-sphere is a closed smooth manifold homotopy
equivalent to \(S^7\). Its oriented diffeomorphism class lies in the group \(\Theta_7\) under connected sum.
Kervaire and Milnor proved that
\[
 \Theta_7\cong\mathbb Z/28\mathbb Z,
\]
with the standard sphere representing the identity~\cite{KM}.
A smooth homotopy seven-sphere is called \emph{exotic}
when it is not diffeomorphic to the standard sphere.

Since the standard sphere carries a metric of constant positive
sectional curvature, the existence of exotic spheres raises a natural
question:
\begin{quote}
\emph{Which smooth structures on \(S^7\) admit metrics of positive
sectional curvature?}
\end{quote}
This question can be asked separately in each of the 28 oriented classes.

Let \(g\) be a Riemannian metric on a smooth manifold \(M\), and write
\(\langle\cdot,\cdot\rangle\) for its inner product. For linearly
independent vectors \(u,v\in T_pM\), the sectional curvature is
\[
\Ksec_p(u,v)=
\frac{\langle R(u,v)v,u\rangle}
{\langle u,u\rangle\langle v,v\rangle-\langle u,v\rangle^2},
\]
where \(R\) is the Riemann curvature tensor with the sign convention
specified below. Let \(\Gr_2(T_pM)\) denote the Grassmannian of
two-planes in \(T_pM\). We also write \(\Ksec_p(\Pi)\), or
\(\Ksec_g(\Pi)\) to emphasize the metric. Positive sectional curvature
means that \(\Ksec_p(\Pi)>0\) for every \(p\in M\) and every
\(\Pi\in\Gr_2(T_pM)\).

For a positively curved metric and \(0<\delta<1\), strict
\(\delta\)-pinching means
\[
\delta<\Ksec_{\min}/\Ksec_{\max}.
\]
For pointwise pinching, the extrema are taken over \(\Gr_2(T_pM)\)
separately at each point \(p\). On a compact manifold, global pinching
uses the extrema over all points and two-planes. The differentiable
sphere theorem of Brendle and Schoen~\cite{BS} states that a closed,
simply connected Riemannian manifold of dimension at least four is
diffeomorphic to a sphere if it is strictly pointwise \(1/4\)-pinched,
that is,
\[
 \Ksec_{\min}(p)>\tfrac14\Ksec_{\max}(p)>0
 \qquad\text{at every point }p.
\]
An exotic seven-sphere therefore admits no strictly pointwise
\(1/4\)-pinched metric. The existence of positive curvature without
this pinching condition is a different question.

Gromoll and Meyer constructed the first nonnegatively curved exotic
sphere in 1974~\cite{GM}; we denote it by \(\Sigma_{\mathrm{GM}}\). Their description as a biquotient of
\(\mathrm{Sp}(2)\) made Lie group geometry and Riemannian submersions
available for its study. Grove and Ziller~\cite{GZ} subsequently
constructed nonnegatively curved metrics on all Milnor spheres, i.e., the
homotopy seven-spheres that arise as \(S^3\)-bundles over \(S^4\).
Goette, Kerin and Shankar~\cite{GKS} extended this existence result
to every exotic seven-sphere.

For the Gromoll--Meyer sphere, Wilhelm~\cite{Wilhelm} obtained positive
sectional curvature almost everywhere, and Eschenburg and
Kerin~\cite{EK} constructed an almost-positively curved metric with an
explicitly described exceptional set. Here, almost positive curvature
means nonnegative sectional curvature everywhere and positive curvature
on every two-plane over an open dense set of points. Strictly positive
constructions on this sphere have been proposed by Petersen and Wilhelm~\cite{PW}, Ouyang~\cite{Ouyang}, and Guo, Fang and Lu~\cite{GFL}.

\subsection{Main results}
Our main theorem is the following.

\begin{theorem}\label{thm:main}
Every smooth homotopy seven-sphere admits a smooth Riemannian metric
with strictly positive sectional curvature. Equivalently, for each
\([\Sigma]\in\Theta_7\), there are a smooth Riemannian metric \(g\)
and a constant \(c_\Sigma>0\) such that, for every \(p\in\Sigma\),
\begin{equation}\label{eq:main-lower-bound}
 \Ksec_g(\Pi)\ge c_\Sigma
 \qquad(\Pi\in\Gr_2(T_p\Sigma)).
\end{equation}
\end{theorem}

\begin{remark}
    The uniform lower bound follows from compactness of the Grassmann bundle.
\end{remark}

To prove Theorem~\ref{thm:main}, we use the two-disk models of
Dur\'an, P\"uttmann and Rigas~\cite{DPR}. For each \(n\in\mathbb Z\),
the model \(\Sigma_n\) represents \(n\bmod28\) in \(\Theta_7\)
and is obtained by attaching two seven-disks \(D_N,D_S\) along a
map \(\phi_n:\partial D_N\to\partial D_S\). In the polar boundary
coordinates, \(\phi_n\) is represented by \(\sigma^n\) for a fixed
diffeomorphism \(\sigma:S^6\to S^6\). Section~\ref{sec:carrier}
specifies this map and the orientation convention. We call the disks
\emph{marked} once these boundary coordinates and the attaching map
are fixed. The following theorem constructs positive metrics compatible
with this prescribed attaching map.

\begin{theorem}[Compatible positive disk metrics]\label{thm:fillings}
For every \(n\in\mathbb Z\), the marked disks \(D_N,D_S\) defining
\(\Sigma_n\) admit smooth Riemannian metrics \(g_N,g_S\) of strictly
positive sectional curvature, including at their boundaries.
Let \(h_N,h_S\) be the induced boundary metrics and \(B_N,B_S\)
the outward second fundamental forms, with convention
\eqref{eq:boundaryconvention}. The marked boundary identification
\(\phi_n:\partial D_N\to\partial D_S\), represented by \(\sigma^n\)
in the boundary coordinates extending over the disks, satisfies
\[
 h_N=\phi_n^*h_S,
 \qquad B_N+\phi_n^*B_S>0.
\]
The last inequality signifies positive definiteness on the common boundary
tangent bundle.
\end{theorem}

\begin{remark}
    %This result clarifies exactly how curvature can constrain smooth structure.
Brendle–Schoen’s sphere theorem says that sufficiently pinched positive sectional curvature forces a simply connected closed manifold to be diffeomorphic to the standard sphere. In particular, the condition
\[
K_{\min}(p)>\tfrac14K_{\max}(p)>0
\]excludes exotic spheres. Our main theorem would establish the complementary statement:
\[
\boxed{\text{Positive sectional curvature alone excludes none of the smooth types on }S^7.}
\]Thus, the quantitative distribution of curvature, rather than its positive sign alone, becomes essential to this smooth rigidity phenomenon.
\end{remark}

\subsection{The construction}\label{subsec:construction}
Fix \(n\in\mathbb Z\). \hyperref[sec:carrier]{Section~\ref*{sec:carrier}} constructs a
principal right \(S^3\)-bundle \(\pi_n:P_n\to S^7\), with right
action \(R_h(p)=ph\), and a smooth free left action
\[
 \star:S^3\times P_n\longrightarrow P_n,\qquad
 (q,p)\longmapsto q\star p,\qquad
 q\star(ph)=(q\star p)h.
\]
The quotient by this second action is the orbit space
\[
 \Sigma_n:=P_n/{\sim_\star},\qquad
 p\sim_\star p'\ \Longleftrightarrow\
 p'=q\star p\ \text{for some }q\in S^3,
\]
with quotient map
\(\varpi_n:P_n\to\Sigma_n\),
\(p\mapsto[p]_\star=\{q\star p:q\in S^3\}\).
We also write \(P_n/S^3_\star\) for this orbit space and call it
the \emph{star quotient} of \(P_n\).

Choose complementary closed polar disks \(U_N,U_S\subset S^7\)
with common boundary. Their preimages are invariant under \(\star\).
For \(\alpha\in\{N,S\}\), set
\[
 P_\alpha:=\pi_n^{-1}(U_\alpha),\qquad
 D_\alpha:=P_\alpha/S^3_\star,\qquad
 \varpi_\alpha:P_\alpha\longrightarrow D_\alpha,\quad
 p\longmapsto[p]_\star.
\]
In the polar trivializations \(P_\alpha\cong D^7\times S^3\),
where \(D^7\) is the closed unit disk in
\(\operatorname{Im}\Hh\oplus\Hh\), the actions are
\[
 (z,u)h=(z,uh),\qquad q\star(z,u)=(qzq^{-1},qu),
\]
where conjugation acts on both components of \(z\).
The maps \([(z,u)]_\star\mapsto u^{-1}zu\) identify
\(D_N,D_S\) smoothly with seven-disks. In their boundary
coordinates, the attaching map
\(\phi_n:\partial D_N\to\partial D_S\) is \(y\mapsto\sigma^n(y)\),
and
\[
 \Sigma_n\cong D_N\cup_{\phi_n}D_S.
\]
Thus the metrics to be glued are on \(D_N,D_S\); we construct them
as quotients of metrics on the ten-dimensional spaces \(P_N,P_S\).

For a fixed \(\alpha\in\{N,S\}\), write \(B=U_\alpha\),
\(P=P_\alpha\), \(D=D_\alpha\),
\(\pi=\pi_n|_{P_\alpha}\), and \(\varpi=\varpi_\alpha\).
Choose a base metric \(g_B\), a smooth principal connection
\(\omega\in\Omega^1(P;\mathfrak{sp}(1))\), and a smooth positive
function \(r:B\to(0,\infty)\).
The connection specifies the splitting
\[
 TP=H\oplus V,\qquad H=\ker\omega,\qquad V=\ker d\pi.
\]
Let \(Q\) be the \(\operatorname{Ad}\)-invariant inner product on
\(\mathfrak{sp}(1)\) for which \(S^3\) has curvature one. The metric
\[
 G=\pi^*g_B+(r\circ\pi)^2Q(\omega,\omega)
\]
makes \(H\) orthogonal to \(V\), lifts \(g_B\) to \(H\), and gives
the fibre over \(b\) the round metric of radius \(r(b)\). We call
\(G\) the \emph{connection metric with fibre radius \(r\)}.

We choose the data so that the \(\star\)-action preserves \(G\).
Put
\[
 \mathcal H^\star_p:=\bigl(T_p(S^3\star p)\bigr)^{\perp_G}
                       =(\ker d\varpi_p)^{\perp_G}.
\]
These are the \emph{star-horizontal} spaces; they are generally
different from the principal-horizontal spaces \(H_p\).
The quotient metric \(g\) on \(D\) is defined by
\[
 g_{[p]_\star}(d\varpi_pX,d\varpi_pY)=G_p(X,Y),
 \qquad X,Y\in\mathcal H^\star_p.
\]
Invariance makes this definition independent of the representative
\(p\), and \(\varpi:(P,G)\to(D,g)\) is a Riemannian submersion.
For \(\Pi\subset T_{[p]_\star}D\), its horizontal lift is
\(\widetilde\Pi=(d\varpi_p|_{\mathcal H^\star_p})^{-1}(\Pi)\).
O'Neill's formula~\cite{ON} gives
\(\Ksec_g(\Pi)\ge\Ksec_G(\widetilde\Pi)\), so positivity on these
lifted two-planes suffices.

\hyperref[sec:cap]{Section~\ref*{sec:cap}} proves the curvature criterion for the
southern construction. Write \(\Omega\) for the curvature of
\(\omega\) and \(D\Omega\) for its covariant derivative, using
\(\nabla^{g_B}\) and the induced adjoint-bundle connection.
With \(g_B,\omega,Q\) fixed, set
\(r_\eps^2=\eps e^{\eps\varphi}\). A sufficiently strong lower bound on
\(-\Hess\varphi\), chosen in terms of \(\Omega,D\Omega\) before
\(\eps\), supplies positive mixed curvature. The Lie-bracket
terms linear in \(\Omega\) have no small fibre factor; on
decomposable bivectors, they couple horizontal and vertical areas
and are absorbed by the positive contributions from those areas.
For a positively curved base and sufficiently small \(\eps\),
the criterion gives \(\Ksec_{G_\eps}>0\) on every two-plane.

\hyperref[sec:south]{Section~\ref*{sec:south}} applies it with a round base disk of radius
less than \(\pi/2\), a multiple of a round height function, and
a \(\star\)-invariant connection that is a product connection in
the boundary trivialization. Theorem~\ref{thm:southfilling} gives
\(g_S\) on \(D_S\).

For \(P_N\cong D^7\times S^3\), \hyperref[sec:north]{Section~\ref*{sec:north}} uses the
flat product connection \(\omega=u^{-1}du\) and the metric
\[
 G_N=ds^2+F(s)^2h_{S^6}+r(s)^2h_{S^3},
\]
where \(s\) is radial distance and \(h_{S^j}\) is the unit round
metric. The function \(r\) is constant near the center and matches
the southern boundary radius and logarithmic derivative. The
\(\varpi_N\)-horizontal space is a graph over the base tangent
space, so its fibre components are bounded in terms of its base
components. This estimate allows the base curvature to dominate
the radial--fibre and angular--fibre terms on these planes.
The angular profile \(F\) is fixed before the common \(\eps\);
Theorem~\ref{thm:northfilling} then gives \(\Ksec_{G_N}>0\)
on the \(\varpi_N\)-horizontal two-planes and hence
\(\Ksec_{g_N}>0\) on \(D_N\).

Finally, let \(h_N,h_S\) be the induced boundary metrics and
\(B_N,B_S\) the outward second fundamental forms of
\((D_N,g_N),(D_S,g_S)\).

\hyperref[sec:gluing]{Section~\ref*{sec:gluing}} compares them
under \(\phi_n\). In the common boundary trivialization, the
source boundary metrics agree and the fibre radii have the same
derivative in a signed normal coordinate across the seam.
The fibre contributions to the outward second fundamental forms
therefore cancel, while the angular contributions have a positive
definite sum. Proposition~\ref{prop:boundarycompatibility} proves
\[
 h_N=\phi_n^*h_S,\qquad B_N+\phi_n^*B_S>0,
\]
which completes Theorem~\ref{thm:fillings}.

The Reiser--Wraith gluing theorem~\cite[Theorem~A(i), \(k=1\)]{RW} then yields a smooth positively curved metric on
\(D_N\cup_{\phi_n}D_S\cong\Sigma_n\).
\hyperref[app:gluing]{Appendix~\ref*{app:gluing}} includes a proof of the strict gluing statement used here.

\subsection{Conventions and notation}\label{subsec:conventions}
Our curvature convention is
\[
 R(A,B)C=\nabla_A\nabla_BC-\nabla_B\nabla_AC-\nabla_{[A,B]}C.
\]
For a decomposable bivector \(\xi=A\wedge B\), put
\[
 \cK(\xi)=\langle R(A,B)B,A\rangle.
\]
Thus \(\Ksec(\operatorname{span}\{A,B\})=\cK(\xi)/|\xi|^2\), and the
unit round sphere has curvature one. Whenever curvature is regarded as
a bilinear form on \(\Lambda^2\), our convention is
\(\mathcal R(A\wedge B,C\wedge D)=-\langle R(A,B)C,D\rangle\).
For a boundary,
\begin{equation}\label{eq:boundaryconvention}
 B(Y,Z)=\langle\nabla_Y n_{\rm out},Z\rangle.
\end{equation}
If \(\rho\) is inward distance and the metric
is \(d\rho^2+h_\rho\), then \(B=-\tfrac12\partial_\rho h_\rho\) at
the boundary.

We identify \(S^3\) with the unit quaternions. Its Lie algebra
\(\operatorname{Im}\Hh\) has the Euclidean inner product \(Q\), so its
bi-invariant metric has sectional curvature one and
\([u,v]=2u\times v\). The symbol \(h_{S^j}\) denotes the unit round
metric on \(S^j\). Frame indices \(i,j,k\) are horizontal and
\(a,b,c\) are vertical, with \(1\le a,b,c\le3\). Repeated frame
indices in displayed contractions are summed over their full ranges
unless a restricted sum is specified.

Exterior powers carry their induced Euclidean norms. In particular,
for an adjoint-bundle-valued two-form we use
\begin{equation}\label{eq:normconvention}
 |\Omega|^2=\sum_{i<j,a}(\Omega_{ij}^a)^2,
 \qquad |D\Omega|^2=\sum_{k,i<j,a}((D_k\Omega)_{ij}^a)^2.
\end{equation}
Consequently their full ordered component arrays have norms
\(\sqrt2|\Omega|\) and \(\sqrt2|D\Omega|\). A mixed tensor
\(\beta\in H\otimes V\) has norm
\(|\beta|^2=\sum_{i,a}\beta_{ia}^2\). Other tensor norms are the induced
Hilbert--Schmidt norms; operator norms are marked explicitly.
Angular and fibre components are taken in orthonormal frames for the
scaled metric. The maps \(K_x\) and \(K_x^*\) introduced in
Section~\ref{sec:north} are measured in the unit-round angular metric
and the Lie algebra inner product \(Q\).

\subsection*{Code}
A reproducible SageMath verification package accompanies this paper, providing exact symbolic checks of algebraic and curvature identities,
coordinate and boundary compatibility formulas, and the arithmetic underlying parameter choices and error estimates, together with rigorous
interval checks of scalar bounds. The source code, exact-arithmetic certificates, independent certificate checker, and a detailed verification report with reproduction instructions
are freely available at \href{https://zliu-math.github.io/projects/PosSecSmoothS7}{https://zliu-math.github.io/projects/PosSecSmoothS7}.

\subsection*{Acknowledgments}
The authors are grateful to AI models, specially ChatGPT 6 Astra and Claude Pro, for valuable assistance with the exploration of some of the proof strategies and calculations. a
That said, we take responsibility for the human verification and exposition.

\section{The marked quaternionic models}\label{sec:carrier}

We describe the principal bundles and star quotients used below, with
fixed trivializations and quotient disk coordinates. We then compare
their attaching maps with the two-disk models of
Dur\'an, P\"uttmann and Rigas~\cite{DPR}. These coordinates will be used in
Sections~\ref{sec:south}--\ref{sec:gluing}.

For each \(n\in\mathbb Z\), we specify a smooth principal
right \(S^3\)-bundle
\[
 \pi_n:P_n\longrightarrow S^7
\]
and two commuting \(S^3\)-actions on its ten-dimensional total space:
the \emph{principal right action} and a smooth free left action,
called the \emph{star action} and denoted by \(\star\). The quotient
by the principal right action is the base \(S^7\). We write
\[
 \Sigma_n:=P_n/S^3_\star
\]
for the orbit space of the star action and give explicit smooth
coordinates identifying it as
\[
 \Sigma_n\cong D_N\cup_{\phi_n}D_S,
\]
where \(D_N,D_S\) are smooth seven-disks and
\(\phi_n:\partial D_N\to\partial D_S\) is the attaching
diffeomorphism computed below. The identification of this quotient
with the class \(n\bmod28\) in \(\Theta_7\) is the conclusion of
\cite[Theorem~1]{DPR}; we verify its application by comparing the
attaching map in our coordinate realization with the one in their Section~4.

The resulting transition formula is used to choose the southern connection potentials with the required behavior near the boundary and the south pole in Section~\ref{sec:south}. The quotient disk coordinates give the boundary markings used in Section~\ref{sec:north}. Their change of coordinates specifies the attaching map under which Section~\ref{sec:gluing} compares the boundary metrics and outward second fundamental forms.

\subsection{Quaternionic coordinates and conjugation}
Let \(\Hh\) denote the algebra of quaternions, regarded as the real
Euclidean space \(\Rr^4\). For \(v\in\Hh\), write \(\bar v\) for its
quaternionic conjugate and \(|v|=(v\bar v)^{1/2}\) for its norm. Set
\[
 \operatorname{Im}\Hh=\{v\in\Hh:\bar v=-v\},
 \qquad S^3=\{q\in\Hh:|q|=1\}.
\]
Thus \(\operatorname{Im}\Hh\) is a three-dimensional real vector
space, and \(S^3\) is a group under quaternionic multiplication, with
\(q^{-1}=\bar q\). The seven-dimensional real vector space
\[
 \mathcal V=\operatorname{Im}\Hh\oplus\Hh
\]
has elements \(x=(p,w)\) with norm
\(|x|^2=|p|^2+|w|^2\). Its unit sphere is
\[
 S^6=\{(p,w)\in\mathcal V:|p|^2+|w|^2=1\}.
\]
Here \(p \in \operatorname{Im}\Hh \) is purely imaginary, whereas \(w\) is an arbitrary
quaternion.

For \(q\in S^3\) and \(x=(p,w)\in\mathcal V\), define simultaneous conjugation by
\[
 C_q(p,w)=(qpq^{-1},qwq^{-1}).
\]
Quaternionic conjugation by \(q\) preserves norms and the imaginary
subspace, so \(C_q\) maps \(S^6\) to itself.

The identities
\(C_1=\Id\) and \(C_{q_1q_2}=C_{q_1}\circ C_{q_2}\) make this a
\textit{left} \(S^3\)-action on $\mathcal V$. We abbreviate \(C_q(x)\) to \(qxq^{-1}\);
products with a pair \(x=(p,w)\) are always taken componentwise. We also let \(S^3\) act on itself by \(v\mapsto qvq^{-1}\).

A map \(f:S^6\to S^3\) is \emph{equivariant} for these actions when
\[
 f(qxq^{-1})=qf(x)q^{-1}
 \qquad(q\in S^3,\ x\in S^6).
\]

\subsection{The equivariant Blakers--Massey map}
For an imaginary quaternion \(p \in \operatorname{Im}\Hh \), the quaternionic exponential is
\[
 \exp(\pi p)=\cos(\pi|p|)
       +\frac{\sin(\pi|p|)}{|p|}\,p,
\]
where the scalar factor \(\sin(\pi|p|)/|p|\) is assigned its limiting
value \(\pi\) at \(p=0\); in particular, \(\exp(0)=1\).
The Blakers--Massey map \(b:S^6\to S^3\) used in
\cite{DMR,DPR} is given, for \(w\ne0\), by
\begin{equation}\label{eq:BM}
b(p,w)=\frac{w}{|w|}\exp(\pi p)\frac{\bar w}{|w|}.
\end{equation}
The factors \(w/|w|\) and \(\exp(\pi p)\) are unit quaternions,
and \(\bar w/|w|=(w/|w|)^{-1}\), so this formula takes values in
\(S^3\).

To define \(b\) at \(w=0\) and verify smoothness, put \(s=|p|^2\).
On \(S^6\) we have \(|w|^2=1-s\), and \eqref{eq:BM} becomes
\[
 b(p,w)=\cos(\pi\sqrt s)+\psi(s)wp\bar w,
 \quad \text{ with }\quad
 \psi(s)=\frac{\sin(\pi\sqrt s)}{\sqrt s(1-s)}
 \quad(0<s<1).
\]
Here \(\psi\) is a real-valued auxiliary function. It extends
smoothly to both endpoints, with
\[
 \psi(0)=\pi,\qquad \psi(1)=\frac\pi2.
\]
Indeed, near \(s=0\),
\[
 \frac{\sin(\pi\sqrt s)}{\sqrt s}
 =\sum_{k=0}^{\infty}
     \frac{(-1)^k\pi^{2k+1}}{(2k+1)!}\,s^k,
\]
while near \(s=1\) its numerator has the expansion
\[
 \sin(\pi\sqrt s)=\frac\pi2(1-s)+O((1-s)^2).
\]
The zero at \(s=1\) therefore removes the factor \(1-s\) from the
denominator. The function \(\cos(\pi\sqrt s)\) is also smooth at
\(s=0\), as its power series contains only integer powers of \(s\).
Consequently the displayed expression for the Blakers--Massey map \(b\) is smooth on all
of \(S^6\). It gives \(b(p,0)=-1\) when \(|p|=1\), and
\(b(0,w)=1\) when \(|w|=1\). Its values remain in \(S^3\) by
continuity.

The identities
\[
 \exp(\pi qpq^{-1})=q\exp(\pi p)q^{-1},
 \qquad \overline{qwq^{-1}}=q\bar wq^{-1}
\]
show that \(b\) is equivariant when \(w\ne0\); continuity gives
the same identity at \(w=0\).

For each \(n\in\mathbb Z\), define the \(n\)-th pointwise power of the Blakers--Massey map
\[
 g_n:S^6\longrightarrow S^3,\qquad g_n(x)=b(x)^n.
\]
The powers here are taken pointwise in the group \(S^3\):
\(g_0\equiv1\), and \(g_{-m}(x)=(b(x)^{-1})^m\) for \(m>0\).
Since multiplication and inversion on \(S^3\) are smooth, each
\(g_n\) is smooth. Equivariance of \(b\) gives
\begin{equation}\label{eq:equiv}
g_n(qxq^{-1})=qg_n(x)q^{-1},
\qquad x\in S^6,\quad q\in S^3.
\end{equation}

\subsection{The principal bundle \texorpdfstring{\(P_n\)}{Pn}}
Use the model
\[
 S^7=\{(x_0,z)\in\Rr\oplus\mathcal V:x_0^2+|z|^2=1\},
\]
with north and south poles \(o_N=(1,0)\) and \(o_S=(-1,0)\).
The polar parameter is \(t=\arccos x_0\in[0,\pi]\). Away from
the poles, write a base point uniquely as
\[
 (x_0,z)=(\cos t,\sin t\,x),
 \qquad 0<t<\pi,\quad x\in S^6.
\]
Let \(\mathcal U_N=S^7\setminus\{o_S\}\) and
\(\mathcal U_S=S^7\setminus\{o_N\}\). These are polar coordinate
neighborhoods: their disk coordinates are \(t x\) and
\((\pi-t)x\), respectively, with the corresponding pole represented
by the zero vector.

Take the product bundles \(\mathcal U_N\times S^3\) and
\(\mathcal U_S\times S^3\). Denote their fibre coordinates by
\(u_N,u_S\in S^3\). For a common base point \((t,x)\) in the
overlap, identify the two fibre coordinates by
\begin{equation}\label{eq:principal}
u_S=g_n(x)u_N.
\end{equation}
Thus \((t,x,u_N)\) and \((t,x,g_n(x)u_N)\) are representatives of
the same point of the glued total space \(P_n\).
Appendix~\ref{app:principal} gives the corresponding quotient-space
construction and its smooth principal-bundle structure.

The projection
\(\pi_n:P_n\to S^7\) sends either representative to
\((\cos t,\sin t\,x)\). The transition map is smooth, with inverse
\(u_N=g_n(x)^{-1}u_S\), so these product charts define a smooth
ten-dimensional bundle with fibre \(S^3\).

\subsection{Two actions of \texorpdfstring{\(S^3\)}{S3} on \texorpdfstring{\(P_n\)}{Pn}}

The \emph{principal right action} is right multiplication in each
fibre. In either chart it is
\[
 (t,x,u)\cdot h=(t,x,uh),\qquad h\in S^3.
\]
It respects \eqref{eq:principal} because
\(g_n(x)(u_Nh)=(g_n(x)u_N)h\). It fixes the base point and is
free and transitive on each fibre: two points of a fibre differ by
right multiplication by a unique \(h\). This is the principal
right \(S^3\)-bundle structure on \(P_n\).

The second action, called the \emph{star action},
\[
 \star:S^3\times P_n\longrightarrow P_n,
 \qquad (q,v)\longmapsto q\star v,
\]
is given in either polar trivialization, away from the poles, by
\begin{equation}\label{eq:star}
 q\star(t,x,u)=(t,qxq^{-1},qu),
 \qquad q\in S^3.
\end{equation}

This is a left action: \(1\star v=v\) and
\(q_1\star(q_2\star v)=(q_1q_2)\star v\). In a polar disk
coordinate \(\zeta=\rho x\), the same formula reads
\[
 q\star(\zeta,u)=(q\zeta q^{-1},qu),
\]
which is smooth also at \(\zeta=0\). To check compatibility with
the transition, use \eqref{eq:equiv}:
\[
 g_n(qxq^{-1})(qu_N)=qg_n(x)u_N=qu_S.
\]
Hence the local formulas define an action on all of \(P_n\).
It is free, since \(qu=u\) implies \(q=1\), including over the
poles. It commutes with the principal right action, since
\[
 q\star\bigl((\zeta,u)\cdot h\bigr)
 =(q\zeta q^{-1},quh)
 =\bigl(q\star(\zeta,u)\bigr)\cdot h.
\]
Unlike the right action, this second action generally changes the
base point: it covers \((x_0,z)\mapsto(x_0,qzq^{-1})\) on \(S^7\).

\subsection{Quotient disks and boundary markings}
For \(v,v'\in P_n\), define
\[
 v\sim_\star v'
 \quad\Longleftrightarrow\quad
 v'=q\star v\ \text{for some }q\in S^3.
\]
We write
\[
 \Sigma_n:=P_n/{\sim_\star}=P_n/S^3_\star,
 \qquad
 \varpi_n:P_n\longrightarrow\Sigma_n,\quad
 v\longmapsto[v]_\star=\{q\star v:q\in S^3\}.
\]
Thus \(S^3_\star\) specifies the action used to form the orbit
space, rather than a different group. The quotient by the principal
right action is the base \(S^7\); the quotient by \(\star\) is
\(\Sigma_n\). We now give explicit smooth coordinates for the
latter quotient.

Choose \(0<a<\pi\), and let \(U_N=\{0\le t\le a\}\) and
\(U_S=\{a\le t\le\pi\}\) be the two closed polar disks in
\(S^7\). The cut value \(a\) will be chosen in
Section~\ref{sec:south}. For \(\alpha\in\{N,S\}\), set
\[
 P_\alpha=\pi_n^{-1}(U_\alpha),\qquad
 D_\alpha=P_\alpha/S^3_\star.
\]
Each \(P_\alpha\) is invariant under \(\star\). Its disk coordinate
is \(\zeta_\alpha=\rho_\alpha x\), where
\(\rho_N=t\) and \(\rho_S=\pi-t\). Put
\(R_N=a\), \(R_S=\pi-a\), and write
\(\mathbb D(R)=\{z\in\mathcal V:|z|\le R\}\).
In the product trivialization of \(P_\alpha\), define
\[
 F_\alpha(\zeta_\alpha,u_\alpha)
       =u_\alpha^{-1}\zeta_\alpha u_\alpha
       \in\mathbb D(R_\alpha).
\]
This map is constant on star orbits, since
\[
 (qu_\alpha)^{-1}(q\zeta_\alpha q^{-1})(qu_\alpha)
       =u_\alpha^{-1}\zeta_\alpha u_\alpha.
\]
Every orbit has a unique representative of the form \((Y,1)\):
acting on \((\zeta,u)\) by \(q=u^{-1}\) gives
\((u^{-1}\zeta u,1)\), and an action preserving the second
coordinate \(1\) must have \(q=1\). Therefore \(F_\alpha\) induces
a diffeomorphism
\[
 D_\alpha\longrightarrow\mathbb D(R_\alpha),\qquad
 [ (\zeta_\alpha,u_\alpha) ]_\star
       \longmapsto Y_\alpha=u_\alpha^{-1}\zeta_\alpha u_\alpha,
\]
with inverse \(Y\mapsto[(Y,1)]_\star\). These formulas are smooth
at the disk center as well as away from it. They identify each
\(D_\alpha\) with a smooth seven-disk; the transition between
the two quotient coordinate systems is computed below.

The boundary of each quotient disk is parameterized by the unit
sphere \(S^6\) using
\[
 y_\alpha=R_\alpha^{-1}Y_\alpha
          =u_\alpha^{-1}xu_\alpha\in S^6.
\]
These are the boundary markings. Let
\(\phi_n:\partial D_N\to\partial D_S\) be the identification
induced by the bundle transition \eqref{eq:principal}; then
\(\Sigma_n\cong D_N\cup_{\phi_n}D_S\).
Let \(\mathbb D_N^7\) and \(\mathbb D_S^7\) be labelled copies of
\(\mathbb D(1)\subset\mathcal V\). Write the normalized quotient
charts as
\[
 \Psi_\alpha:D_\alpha\longrightarrow\mathbb D_\alpha^7,
 \qquad
 \Psi_\alpha([ (\zeta_\alpha,u_\alpha) ]_\star)
       =R_\alpha^{-1}u_\alpha^{-1}\zeta_\alpha u_\alpha,
 \qquad \alpha\in\{N,S\},
\]
and let \(m_\alpha:=\Psi_\alpha|_{\partial D_\alpha}\) denote their
boundary restrictions. These are the markings
\(y_\alpha=u_\alpha^{-1}xu_\alpha\) defined above.
The following lemma computes the attaching map in these coordinates.
This is the coordinate comparison needed to apply
\cite[Theorem~1 and its proof in Section~4]{DPR}.

\begin{lemma}[The attaching map]\label{lem:topology}
The map
\[
 \sigma:S^6\longrightarrow S^6,\qquad
 \sigma(x)=b(x)^{-1}xb(x),
\]
is a diffeomorphism, and the boundary markings satisfy
\[
 m_S\circ\phi_n\circ m_N^{-1}=\sigma^n.
\]
The quotient charts \(\Psi_N,\Psi_S\) induce a smooth diffeomorphism
\[
 J_n:\Sigma_n\longrightarrow
 M_n:=\mathbb D_N^7\cup_{\sigma^n}\mathbb D_S^7.
\]
\end{lemma}
\begin{proof}
For representatives of the same point over the common boundary,
\eqref{eq:principal} gives \(u_S=g_n(x)u_N\). By equivariance,
\(g_n(u_N^{-1}xu_N)=u_N^{-1}g_n(x)u_N\). Thus
\[
\begin{aligned}
 y_S
 &=u_N^{-1}g_n(x)^{-1}xg_n(x)u_N\\
 &=g_n(y_N)^{-1}y_Ng_n(y_N).
\end{aligned}
\]

Set \(\widehat\sigma(x)=b(x)xb(x)^{-1}\). Equivariance of \(b\)
gives
\[
 b(\sigma(x))=b(x),\qquad b(\widehat\sigma(x))=b(x).
\]
Both maps are smooth, and substitution yields
\(\widehat\sigma\circ\sigma
 =\sigma\circ\widehat\sigma=\Id_{S^6}\).
Hence \(\sigma\) is a diffeomorphism. Iterating \(\sigma\) and its
inverse, respectively, gives
\[
 \sigma^n(x)=b(x)^{-n}xb(x)^n=g_n(x)^{-1}xg_n(x)
 \qquad(n\in\mathbb Z).
\]
Substitution in the boundary calculation gives
\begin{equation}\label{eq:clutching}
 y_S=\sigma^n(y_N),
\end{equation}
which is precisely \(m_S\circ\phi_n\circ m_N^{-1}=\sigma^n\).

The charts \(\Psi_N,\Psi_S\) therefore descend to the stated
bijection \(J_n\). Their smooth inverses on the individual disks,
\(y\mapsto[(R_\alpha y,1)]_\star\), agree on the identified
boundaries: for \(y\in S^6\),
\[
 \bigl[(R_Ny,1)_N\bigr]_\star
 =\bigl[(R_Sy,g_n(y))_S\bigr]_\star
 =\bigl[(R_S\sigma^n(y),1)_S\bigr]_\star,
\]
where the subscripts indicate the product trivializations.
On the open polar overlap, the unnormalized quotient coordinates
satisfy
\[
 Y_S=(\pi-|Y_N|)\,
        \sigma^n\!\left(\frac{Y_N}{|Y_N|}\right),
 \qquad 0<|Y_N|<\pi.
\]
To check smoothness across the seam at \(t=a\), let \(\rho\) be the
radial coordinate in either unit disk. Use
\(R_N(\rho-1)\) on the northern collar and
\(R_S(1-\rho)\) on the southern collar as the signed collar
parameter; both pull back to \(t-a\). Use \(y_N\) on the northern
side and \(\sigma^{-n}(y_S)\) on the southern side as the angular
coordinate. In these collars, \(J_n\) and its inverse preserve both
coordinates, so they are smooth across the seam.
\end{proof}

The proof of \cite[Theorem~1, Section~4]{DPR} presents their sphere
\(\Sigma_n^{\mathrm{DPR}}\) as two polar disks attached by the
\(n\)-th iterate of
\[
 (p,w)\longmapsto\overline{b(p,w)}\,(p,w)\,b(p,w).
\]
This is the same \(\sigma\), since
\(\overline{b(p,w)}=b(p,w)^{-1}\).
To compare the cut \(t=a\) with their cut at \(t=\pi/2\), use a
smooth increasing diffeomorphism of \([0,\pi]\) fixing its endpoints,
carrying \(a\) to \(\pi/2\), and linear near the endpoints.
This radial change preserves the angular attaching map and is
smooth at the disk centers and across the seam. It identifies
\(M_n\) with \(\Sigma_n^{\mathrm{DPR}}\).
Give \(M_n\) the orientation induced by this comparison and
orient \(\Sigma_n\) by pullback through \(J_n\). Applying
\cite[Theorem~1]{DPR} then gives
\[
 [\Sigma_n]=[\Sigma_n^{\mathrm{DPR}}]
          =n[\Sigma_{\mathrm{GM}}]
 \quad\text{in }\Theta_7\cong\mathbb Z/28\mathbb Z,
\]
with the orientation of \(\Sigma_{\mathrm{GM}}\) chosen to represent
\(1\). Here \(\Theta_7\) denotes the group of oriented diffeomorphism
classes of homotopy seven-spheres. Since its group operation is
oriented connected sum,
\[
 \Sigma_2\cong_{\mathrm{or}}
 \Sigma_{\mathrm{GM}}\#\Sigma_{\mathrm{GM}},
\]
where \(\cong_{\mathrm{or}}\) denotes an orientation-preserving
diffeomorphism.

If the star action is forgotten, the principal bundle \(P_n\) is
classified by the homotopy class \([g_n]=n[b]\) of its transition
map in \(\pi_6(S^3)\cong\mathbb Z/12\mathbb Z\). A principal-bundle
isomorphism respects the right action but need not respect the star
action. Thus isomorphic principal bundles can have different star
quotients; in particular, some exotic quotients arise from trivial
principal bundles. This distinction is described in
\cite[Corollary~2.2 and Remark~2.5]{DPR}.

\begin{remark}[Boundary coordinates]\label{rem:boundarycoordinates}
On a collar of the southern boundary, we use the north product
trivialization. Its quotient angular coordinate is \(y_N\), whereas
the extending south trivialization gives \(y_S\).
Equation~\eqref{eq:clutching} relates these angular coordinates
throughout the overlap. Hence the boundary identification is the
identity when both sides are expressed in the common north
trivialization, and is represented by \(\sigma^n\) in the extending
north and south disk markings. Section~\ref{sec:gluing} compares
the boundary metrics and second fundamental forms under this map.
\end{remark}

\section{Positive connection metrics over a disk}\label{sec:cap}

The southern construction requires a connection metric with positive
sectional curvature. The next proposition gives a sufficient
condition in terms of the Hessian of the fibre radius and the
curvature of a fixed connection.

Let \(\pi:P\to(B,g_B)\) be a principal right \(S^3\)-bundle, with
a fixed smooth connection \(\omega\) and curvature \(\Omega\).
For a positive function \(r\) on \(B\), consider
\begin{equation}\label{eq:connectionmetric}
G=\pi^*g_B+r^2Q(\omega,\omega).
\end{equation}
In this section, \(H\) and \(V\) denote the horizontal and vertical
distributions of the principal projection \(\pi\), respectively.

\Needspace{12\baselineskip}
\begin{proposition}[Positive connection metrics]\label{prop:concave}
Suppose that \(B\) is compact, possibly with boundary, and
\(\Ksec_{g_B}\ge\kappa>0\). Put
\[
M_0=\sup_B|\Omega|,\qquad
M_1=\sup_B|D\Omega|,
\]
where \(D\) uses the base Levi--Civita connection and the induced
adjoint-bundle connection, and the norms are \eqref{eq:normconvention}.
Set \(C=4\). If
\begin{equation}\label{eq:Lambda}
\Lambda\ge16CM_0^2+
           64C^2\kappa^{-1}(M_1+1)^2+4
\end{equation}
and a fixed smooth function \(\varphi\) satisfies
\[
-\Hess_B\varphi\ge\Lambda g_B,
\]
then the metrics
\begin{equation}\label{eq:smallradius}
G_\eps=\pi^*g_B+r_\eps^2Q(\omega,\omega),
\qquad r_\eps^2=\eps e^{\eps\varphi},
\end{equation}
have strictly positive sectional curvature for all sufficiently small
\(\eps>0\).
\end{proposition}

We first compute the curvature of \eqref{eq:connectionmetric} and
estimate it on decomposable bivectors. The concavity condition
\eqref{eq:Lambda} is imposed before choosing \(\eps\): after
Young's inequality, the terms involving \(D\Omega\) contribute a
loss of order \(\eps(M_1+1)^2\), the same order in \(\eps\) as
the positive Hessian term. On the southern disk, \(\varphi\) will
be a multiple of a round height function.

\begin{remark}[Mixed curvature and the fibre radius]\label{rem:flatconnection}
Variation of the fibre radius can give positive mixed curvature even
for a flat connection. For example,
\[
d\rho^2+\sin^2\rho\,h_{S^6}+\cos^2\rho\,h_{S^3},
\qquad 0\le\rho\le R<\pi/2,
\]
is the round metric of curvature one on a tube around a great
\(S^3\subset S^{10}\). Viewed as a connection metric over a
seven-disk, it has a flat connection and mixed curvature
\(-r^{-1}\Hess r=g_B\).
\end{remark}

\subsection{The curvature components}
Choose a base orthonormal frame normal at a given point, and let
\(X_i\) be its principal-horizontal lifts. For a fixed \(Q\)-orthonormal
basis \(e_a\) of \(\operatorname{Im}\Hh\), write
\[
E_a=r^{-1}e_a^\#,\qquad \vartheta=d\log r,\qquad
[e_a,e_b]=c_{ab}^{\ c}e_c.
\]
The fields \(e_a^\#\) are the fundamental fields of the principal
right action, with
\(e_a^\#(p)=\left.\frac{d}{dt}\right|_{t=0}p\exp(te_a)\)
for \(p\in P\). At the chosen point, the relevant bracket identities are
\begin{align}
[X_i,X_j]^\ver&=-r\Omega_{ij}^aE_a,&
[X_i,E_a]&=-\vartheta_iE_a,&
[E_a,E_b]&=r^{-1}c_{ab}^{\ c}E_c.
\label{eq:brackets}
\end{align}
Since the curvature form is adjoint-equivariant,
\begin{equation}\label{eq:verticalderivative}
E_a(\Omega_{ij}^b)=-r^{-1}c_{ac}^{\ b}\Omega_{ij}^c.
\end{equation}
Koszul's formula gives
\begin{align}
\nabla_{X_i}X_j
  &=(\nabla_i^B X_j)^\hor-\frac r2\Omega_{ij}^aE_a,
&\nabla_{X_i}E_a&=\frac r2\Omega_{ij}^aX_j,
\label{eq:Koszul1}\\
\nabla_{E_a}X_i
  &=\frac r2\Omega_{ij}^aX_j+\vartheta_iE_a,
&\nabla_{E_a}E_b&=\frac1{2r}c_{ab}^{\ c}E_c-\delta_{ab}\vartheta_iX_i.
\label{eq:Koszul2}
\end{align}
We differentiate these identities in the chosen frame and then
evaluate at the base point.

Set \(N=-r^{-1}\Hess_Br\), viewed as a symmetric endomorphism using
\(g_B\). We write
\(\Omega(X_i,X_j)=\Omega_{ij}^a e_a\) in the fixed
\(Q\)-orthonormal Lie algebra basis. The curvature components
involving both distributions are
\begin{align}
\langle R(X_i,E_a)E_b,X_j\rangle
 &=N_{ij}\delta_{ab}
   +\frac{r^2}{4}\sum_k\Omega_{ik}^b\Omega_{jk}^a
   -\frac14\langle[e_a,e_b],\Omega_{ij}\rangle,
\label{eq:HVHV}\\
\langle R(X_i,X_j)E_a,E_b\rangle
 &=\frac{r^2}{4}\sum_k
       (\Omega_{ik}^a\Omega_{jk}^b-\Omega_{jk}^a\Omega_{ik}^b)
   +\frac12\langle[e_a,e_b],\Omega_{ij}\rangle,
\label{eq:HHVV}\\
\langle R(X_i,E_a)E_b,E_c\rangle
 &=\frac r2\sum_k \vartheta_k
       (\delta_{ab}\Omega_{ik}^c-\delta_{ac}\Omega_{ik}^b),
\label{eq:HVVV}\\
\langle R(X_i,X_j)X_k,E_a\rangle
 &=-\frac r2\bigl((D_i\Omega)_{jk}^a-(D_j\Omega)_{ik}^a\bigr)
       \notag\\
 &\hspace{5mm}-\frac r2(\vartheta_i\Omega_{jk}^a-\vartheta_j\Omega_{ik}^a)
       +r\vartheta_k\Omega_{ij}^a.
\label{eq:HHHV}
\end{align}

For \eqref{eq:HVHV}, the relevant horizontal components at the
chosen point are
\begin{align*}
\langle\nabla_{X_i}\nabla_{E_a}E_b,X_j\rangle
  &=\tfrac14c_{ab}^{\ c}\Omega_{ij}^c
       -\delta_{ab}\partial_i\vartheta_j,\\
\langle\nabla_{E_a}\nabla_{X_i}E_b,X_j\rangle
  &=-\tfrac12c_{ac}^{\ b}\Omega_{ij}^c
       +\tfrac{r^2}{4}\sum_k\Omega_{ik}^b\Omega_{kj}^a,\\
\langle\nabla_{[X_i,E_a]}E_b,X_j\rangle
  &=\delta_{ab}\vartheta_i\vartheta_j.
\end{align*}
Use \(c_{ac}^{\ b}=-c_{ab}^{\ c}\),
\(\Omega_{kj}=-\Omega_{jk}\), and
\(N_{ij}=-\partial_i\vartheta_j-\vartheta_i\vartheta_j\) to obtain
\eqref{eq:HVHV}. If its left side is denoted by \(M_{ijab}\),
the first Bianchi identity gives
\[
\langle R(X_i,X_j)E_a,E_b\rangle=M_{jiab}-M_{ijab},
\]
which proves \eqref{eq:HHVV}.

For \eqref{eq:HVVV}, differentiating the vertical term
\(c_{ab}^{\ c}/(2r)\) in \(\nabla_{E_a}E_b\) contributes
\(-\vartheta_i c_{ab}^{\ c}/(2r)\). The bracket term
\(-\nabla_{[X_i,E_a]}E_b\) contributes its opposite. The remaining
vertical components are
\[
\frac r2\delta_{ab}\sum_k \vartheta_k\Omega_{ik}^c
 \quad\hbox{and}\quad
-\frac r2\delta_{ac}\sum_k \vartheta_k\Omega_{ik}^b.
\]
Their sum gives \eqref{eq:HVVV}. For \eqref{eq:HHHV}, the vertical parts of
\(\nabla_{X_i}\nabla_{X_j}X_k-\nabla_{X_j}\nabla_{X_i}X_k\)
are the first two terms of \eqref{eq:HHHV}. The vertical part of
\(-\nabla_{[X_i,X_j]}X_k\) is \(r\vartheta_k\Omega_{ij}^aE_a\),
giving its last term.

The purely horizontal sectional numerator is
\begin{equation}\label{eq:HHHH}
\cK_G(X\wedge Y)=\cK_B(X\wedge Y)
                  -\frac34r^2|\Omega(X,Y)|^2.
\end{equation}
This is the horizontal submersion formula \cite{ON}, since the vertical bracket
is \(-r\Omega(X,Y)^aE_a\). The fibres are round with radius \(r\);
their second fundamental form is
\(-G|_{\ver}\otimes\nabla\log r\), by \eqref{eq:Koszul2}.
Their Gauss equation therefore gives the vertical curvature tensor of
constant-curvature type with coefficient
\begin{equation}\label{eq:VVVV}
r^{-2}-|\vartheta|^2.
\end{equation}
We now combine these components to estimate the sectional curvature
of an arbitrary two-plane.

\subsection{A lower bound for arbitrary planes}
Write a decomposable bivector as
\[
\xi=(X+U)\wedge(Y+V)=\alpha+\beta+\gamma,
\]
where
\[
\alpha=X\wedge Y,\qquad
\beta=X\otimes V-Y\otimes U,\qquad
\gamma=U\wedge V.
\]
Here \(X,Y\) are principal-horizontal and \(U,V\) are
principal-vertical. We identify \(H\wedge V\) with \(H\otimes V\)
by \(X\wedge U\mapsto X\otimes U\). All components are taken in
the normalized frames above. In expressions involving \([U,V]\),
the vertical coefficient vectors are identified with
\(\operatorname{Im}\Hh\) by \(E_a\leftrightarrow e_a\). Set
\[
h=|\alpha|,\qquad m=|\beta|,\qquad k=|\gamma|.
\]
Then \(|\xi|^2=h^2+m^2+k^2\).

The terms linear in \(\Omega\) are estimated using the
skew-symmetry of the Lie bracket. Since \(\xi\) is decomposable,
its components satisfy
\begin{equation}\label{eq:Plucker}
\beta_{ia}\beta_{jb}-\beta_{ib}\beta_{ja}
       =\alpha_{ij}\gamma_{ab}.
\end{equation}
Put \(\mathcal J=\langle\Omega(X,Y),[U,V]\rangle\).
The ordered-index contraction is
\[
 \sum_{i,j,a,b,c}c_{ab}^{\ c}\Omega_{ij}^c\beta_{ia}\beta_{jb}
 =\frac12\sum_{i,j,a,b,c}
       c_{ab}^{\ c}\Omega_{ij}^c\alpha_{ij}\gamma_{ab}
 =2\mathcal J.
\]
The first equality uses skew-symmetry in \(a,b\) and
\eqref{eq:Plucker}; the second uses both skew index pairs.
Consequently the bracket part of the \(\beta\beta\) term in
\eqref{eq:HVHV} is \(-\mathcal J/2\).
For the off-diagonal block, our curvature-form convention and
\eqref{eq:HHVV} give
\[
 2\sum_{\substack{i<j\\a<b}}\alpha_{ij}\gamma_{ab}
       \left(-\frac12\sum_c c_{ab}^{\ c}\Omega_{ij}^c\right)
 =-\mathcal J.
\]
Here the leading factor two is the cross-block factor, and the
minus sign comes from
\(\mathcal R(A\wedge B,C\wedge D)=-\langle R(A,B)C,D\rangle\).
The sum of the two contributions is therefore

\begin{equation}\label{eq:bracketcontraction}
-\frac32\langle\Omega(X,Y),[U,V]\rangle.
\end{equation}
Since \(|[U,V]|=2|U\wedge V|\) and
\(|\Omega(X,Y)|\le M_0|X\wedge Y|\), the absolute value of
\eqref{eq:bracketcontraction} is at most \(3M_0hk\).

The \(N\)-term in \eqref{eq:HVHV} contributes
\[
\sum_{i,j,a}N_{ij}\beta_{ia}\beta_{ja}
       \ge\lambda_{\min}(N)m^2.
\]
For the remaining blocks, we use Cauchy--Schwarz with the norms
in \eqref{eq:normconvention}. The purely horizontal term is at least
\((\kappa-\tfrac34r^2M_0^2)h^2\); only the sectional lower bound of
the base is used, since \(\alpha=X\wedge Y\) is decomposable.
The purely vertical contribution is exactly
\((r^{-2}-|\vartheta|^2)k^2\).

For the quadratic part of \eqref{eq:HVHV}, Cauchy--Schwarz gives
\[
 \left(\sum_{i,j,a,b}
   \left(\sum_k\Omega_{ik}^b\Omega_{jk}^a\right)^2\right)^{1/2}
 \le\sum_{i,k,a}(\Omega_{ik}^a)^2=2|\Omega|^2.
\]
Its contribution therefore has absolute value at most
\(\tfrac12r^2M_0^2m^2\). The full ordered-array norm of the quadratic
part of \eqref{eq:HHVV} is at most \(r^2M_0^2\). Restricting both
skew pairs to \(i<j\) and \(a<b\) divides that norm by two. The
factor two from the off-diagonal block in the curvature quadratic
form then gives the bound \(r^2M_0^2hk\).

The full ordered-array norm of \eqref{eq:HHHV} is bounded by
\[
 \sqrt2rM_1+2\sqrt2r|\vartheta|M_0.
\]
Restricting its first skew pair to \(i<j\) divides this bound by
\(\sqrt2\). Including the cross-block factor two, its contribution
is bounded by \(2r(M_1+2|\vartheta|M_0)hm\).
For \eqref{eq:HVVV}, put \(w_i^a=\sum_k\vartheta_k\Omega_{ik}^a\).
Then
\[
 \sum_{i,a}(w_i^a)^2\le2|\vartheta|^2|\Omega|^2,
 \qquad
 \sum_{a,b,c}(\delta_{ab}w_i^c-\delta_{ac}w_i^b)^2=4|w_i|^2.
\]
After restricting \(b<c\) and including the cross-block factor two,
its contribution is at most \(2r|\vartheta|M_0km\).
Combining the three diagonal and three off-diagonal blocks in
\[
 \Lambda^2(H\oplus V)=\Lambda^2H\oplus(H\otimes V)\oplus\Lambda^2V
\]
with the bracket bound \(3M_0hk\) gives
\begin{align}
 \cK_G(\xi)\ge{}&
  (\kappa-\tfrac34r^2M_0^2)h^2
  +(\lambda_{\min}N-\tfrac12r^2M_0^2)m^2
  +(r^{-2}-|\vartheta|^2)k^2\notag\\
 &-2r(M_1+2|\vartheta|M_0)hm
   -(3M_0+r^2M_0^2)hk-2r|\vartheta|M_0km.
 \label{eq:sharp-master}
\end{align}
In particular, setting \(C=4\) yields the convenient common-constant
estimate
\begin{align}
\cK_G(\xi)\ge{}&
(\kappa-Cr^2M_0^2)h^2+(r^{-2}-|\vartheta|^2)k^2
 +(\lambda_{\min}N-Cr^2M_0^2)m^2\notag\\
&-Cr(M_1+|\vartheta|M_0)hm
 -C(M_0+r^2M_0^2)hk-Cr|\vartheta|M_0km.
\label{eq:master}
\end{align}
The connection enters this estimate through \(M_0\) and \(M_1\).

\begin{remark}[Decomposability]\label{rem:decomposability}
Identity~\eqref{eq:Plucker} converts the terms linear in \(\Omega\)
into a coupling of the horizontal and vertical areas. This is why
\eqref{eq:master} contains \(M_0hk\). The horizontal term is
estimated using the decomposable bivector \(X\wedge Y\), so a
sectional-curvature lower bound on the base suffices throughout.
\end{remark}

\subsection{Choice of the fibre scale}
\begin{proof}[Proof of Proposition~\ref{prop:concave}]
Let \(D_0=\sup_B|d\varphi|\). For \(r^2=\eps e^{\eps\varphi}\),
\[
\vartheta=\frac\eps2d\varphi,\qquad
N=-\frac\eps2\Hess\varphi
        -\frac{\eps^2}{4}d\varphi\otimes d\varphi.
\]
Choose \(\eps>0\) so small that
\begin{gather}
|\eps\varphi|\le\log2,\qquad
2\eps M_0\le1,\qquad \eps D_0M_0\le1,
\label{eq:smallA}\\
8C\eps M_0^2\le\kappa,\qquad
64C^2\eps M_0^2\le\kappa,
\label{eq:smallB}\\
\eps D_0^2\le\Lambda/4,\qquad
2\eps^3D_0^2\le1,\qquad
32C^2\eps^3D_0^2M_0^2\le\Lambda.
\label{eq:smallC}
\end{gather}
All quantities except \(\eps\) have already been fixed. Thus these
finitely many upper bounds admit a common positive solution.

We have \(\eps/2\le r^2\le2\eps\) and \(|\vartheta|\le\eps D_0/2\).
The first two coefficients in \eqref{eq:master} are consequently at
least \(3\kappa/4\) and \(3/(8\eps)\). Young's inequality gives
\begin{align*}
Cr(M_1+|\vartheta|M_0)hm
 &\le\frac\kappa8h^2+
       8C^2\kappa^{-1}\eps(M_1+1)^2m^2,\\
C(M_0+r^2M_0^2)hk
 &\le\frac\kappa8h^2+
       8C^2\kappa^{-1}M_0^2k^2
 \le\frac\kappa8h^2+\frac1{8\eps}k^2,\\
Cr|\vartheta|M_0km
 &\le\frac1{8\eps}k^2+
       2C^2\eps^4D_0^2M_0^2m^2.
\end{align*}
The mixed coefficient starts at \(\eps\Lambda/2\).
The loss from \(d\varphi\otimes d\varphi\) is at most
\(\eps\Lambda/16\); the loss \(Cr^2M_0^2\) is at most
\(\eps\Lambda/8\). The two Young losses in \(m^2\) are at most
\(\eps\Lambda/8\) and \(\eps\Lambda/16\), respectively.
Using \eqref{eq:Lambda} and \eqref{eq:smallC}, we obtain
\begin{equation}\label{eq:reserve}
\cK_{G_\eps}(\xi)\ge
 \frac\kappa2h^2+\frac{\eps\Lambda}{8}m^2
                  +\frac1{8\eps}k^2.
\end{equation}
For a unit decomposable bivector, \(h^2+m^2+k^2=1\). Hence
\begin{equation}\label{eq:southlowerbound}
 \Ksec_{G_\eps}\ge
 c_S:=\min\{\kappa/2,\eps\Lambda/8,1/(8\eps)\}>0.
\end{equation}
For each fixed admissible \(\eps>0\), the right-hand side is positive
and independent of the point and two-plane.
\end{proof}

\section{The southern filling}\label{sec:south}

The southern filling determines the boundary radii and fibre slope
that the northern filling will match. We state the construction on
the marked southern bundle of Section~\ref{sec:carrier}.
Fix \(n\in\mathbb Z\) and choose
\[
\frac\pi2<a<b_0<\pi.
\]
The southern base is the round disk
\[
U_S=\{a\le t\le\pi\},\qquad
g_B=dt^2+\sin^2t\,h_{S^6}.
\]
Its radius, measured from the south pole, is \(\pi-a<\pi/2\).
Choose a fixed smooth function \(\chi(t)\in[0,1]\) which is zero
near \(a\) and one for \(t\ge b_0\).

\Needspace{12\baselineskip}
\begin{theorem}[Southern filling]\label{thm:southfilling}
Let \(\omega\) be the connection specified by
\eqref{eq:potentials}--\eqref{eq:omega}, and let \(M_0,M_1\) be its
curvature bounds as in Proposition~\ref{prop:concave}.
Choose \(\Lambda\) satisfying \eqref{eq:Lambda} with \(\kappa=1\),
and fix \(A_0>0\) such that
\[
A_0|\cos a|\ge\Lambda.
\]
There is \(\eps_S>0\), depending only on these fixed data, such that
for every \(0<\eps<\eps_S\) the metric
\begin{equation}\label{eq:south}
G_S=dt^2+\sin^2t\,h_{S^6}+r_S(t)^2Q(\omega,\omega),
\qquad r_S(t)^2=\eps e^{-\eps A_0\cos t}
\end{equation}
extends smoothly over the entire southern bundle and has strictly
positive sectional curvature. The star action is free and isometric,
and its quotient \((D_S,g_S)\) is a smooth marked seven-disk. Both
\(G_S\) and \(g_S\) satisfy \eqref{eq:southlowerbound} with
\(\kappa=1\), including at the center and boundary.

Set
\begin{equation}\label{eq:boundarydata}
F_a=\sin a,\qquad
r_a^2=\eps e^{\eps A_0|\cos a|},\qquad
q_s=\frac{r_S'(a)}{r_a}=\frac{\eps A_0F_a}{2},
\qquad \mu_S=\cot a<0.
\end{equation}
In the north boundary gauge, \(\omega=u^{-1}du\) on a boundary
collar, and the induced source boundary metric is
\[
F_a^2h_{S^6}+r_a^2h_{S^3}.
\]
For a source boundary tangent vector \(Z=X+U\), expressed in
orthonormal angular and fibre components, the outward second
fundamental form \(\mathcal B_S\) is
\begin{equation}\label{eq:southsourceform}
 \mathcal B_S(Z,Z)=-\mu_S|X|^2-q_s|U|^2.
\end{equation}
\end{theorem}

We construct the connection first, apply the curvature estimate, and
then compute the boundary data in the north gauge.

\subsection{The connection}
In the north boundary gauge and the extending south gauge, respectively,
define connection potentials
\begin{equation}\label{eq:potentials}
A_N=\chi g_n^{-1}dg_n,\qquad
A_S=(\chi-1)dg_n\,g_n^{-1}.
\end{equation}
They satisfy
\[
A_N=g_n^{-1}A_Sg_n+g_n^{-1}dg_n,
\]
which is precisely the transformation rule associated with
\eqref{eq:principal}. The full principal connection form in either
gauge is
\begin{equation}\label{eq:omega}
\omega=u^{-1}Au+u^{-1}du.
\end{equation}
Since \(A_S=0\) near the south pole, this defines a smooth connection
on the entire southern bundle.

Put \(\theta=g_n^{-1}dg_n\) and write
\((\theta\wedge\theta)(X,Y)=[\theta(X),\theta(Y)]\). The curvature
in the north gauge is
\[
\Omega_N=d\chi\wedge\theta+
                     \chi(\chi-1)\theta\wedge\theta.
\]
The connection is smooth on the compact base \(U_S\), so
\(M_0,M_1\) are finite. These constants depend on \(n,a,b_0,\chi\)
and the fixed round base metric.

The connection is invariant under the star action. In fact,
\eqref{eq:equiv} gives \(A\mapsto qAq^{-1}\) under conjugation of
the base, while \(u\mapsto qu\). For constant \(q\),
\[
(qu)^{-1}(qAq^{-1})(qu)+(qu)^{-1}d(qu)
                   =u^{-1}Au+u^{-1}du.
\]
The round base and radial functions are also invariant.

\subsection{The metric and its quotient}
The function \(\varphi(t)=-A_0\cos t\) satisfies
\[
-\Hess_B\varphi=-A_0\cos t\,g_B
                         \ge A_0|\cos a|\,g_B.
\]
Choose \(\eps_S>0\) so that \eqref{eq:smallA}--\eqref{eq:smallC}
hold with \(\kappa=1\) for every \(0<\eps<\eps_S\).
Proposition~\ref{prop:concave} then gives positive sectional curvature
for \(G_S\), with the lower bound \(c_S\) in
\eqref{eq:southlowerbound}.

To express the metric at the south pole, set \(\rho=\pi-t\).
The radius becomes
\[
 \widehat r_S(\rho):=r_S(\pi-\rho)
          =\sqrt\eps\,e^{\eps A_0\cos\rho/2},
\]
a smooth even function. Write
\(r_{\rm pole}:=\widehat r_S(0)=\sqrt\eps\,e^{\eps A_0/2}\).
The connection potential vanishes in the extending south gauge there. At the pole,
\(\Omega=0\), \(\vartheta=d\log r_S=0\), and
\(N=(\eps A_0/2)g_B\); consequently the sectional numerator has
the form
\[
h^2+\frac{\eps A_0}{2}m^2+r_{\rm pole}^{-2}k^2>0
\]
for every nonzero decomposable bivector.

Equip the star quotient \(D_S\) with its induced metric \(g_S\).
The action is free and isometric. For horizontal lifts of tangent
vectors, the submersion formula \cite{ON} reads
\begin{equation}\label{eq:ONeill}
\cK_{g_S}(X\wedge Y)
 =\cK_{G_S}(\widetilde X\wedge\widetilde Y)
          +\frac34|[\widetilde X,\widetilde Y]^{\ver_\star}|^2.
\end{equation}
Here the lifts are extended as basic horizontal fields for the star
quotient, and \(\ver_\star\) denotes orthogonal projection onto the
star-vertical space \(\ker d\varpi_S\). Thus \(D_S\) is a smooth positively curved compact disk,
including its boundary and center.

\subsection{The boundary data}
The connection potential \(A_N\) vanishes on a whole boundary
collar. Thus \eqref{eq:south} is a doubly warped product in that
gauge. Evaluating the radii and their logarithmic derivatives at
\(t=a\) gives \eqref{eq:boundarydata} and the boundary metric in
Theorem~\ref{thm:southfilling}. The outward normal is
\(-\partial_t\), so differentiating the two warped factors gives
\eqref{eq:southsourceform}.
The ratio \(q_s/r_a^2\) tends to \(A_0F_a/2\) as
\(\eps\to0\). The northern angular profile will be chosen in terms
of \(A_0\) and \(F_a\) to accommodate this boundary slope.

\begin{proof}[Proof of Theorem~\ref{thm:southfilling}]
The connection constructed above is smooth and star invariant.
The height function \(\varphi=-A_0\cos t\) satisfies the hypotheses
of Proposition~\ref{prop:concave}, so the choice of \(\eps_S\)
gives the asserted source curvature bound. The pole calculation and
O'Neill's formula \eqref{eq:ONeill} give smoothness and the same
positive lower bound on the quotient. Its disk identification is the
one in Section~\ref{sec:carrier}. The boundary calculation above
proves the remaining assertions.
\end{proof}

\section{The northern filling}\label{sec:north}

We now construct a filling with the boundary radii and fibre slope
provided by Theorem~\ref{thm:southfilling}. Fix
\(a\in(\pi/2,\pi)\) and \(A_0>0\), and define \(F_a,r_a,q_s\)
for each \(\eps>0\) by \eqref{eq:boundarydata}.
In the product trivialization \(P_N\cong D^7\times S^3\), consider
\begin{equation}\label{eq:north}
G_N=ds^2+F(s)^2h_{S^6}+r(s)^2h_{S^3},
\qquad 0\le s\le\ell_N.
\end{equation}
The angular boundary coordinates are those of
Section~\ref{sec:carrier}; the connection is flat, and the star
action is \eqref{eq:star}.

\Needspace{12\baselineskip}
\begin{theorem}[Northern filling with prescribed boundary data]\label{thm:northfilling}
There are \(\ell_N>0\), \(\eps_N>0\) and
\(F\in C^\infty([0,\ell_N])\), chosen independently of \(\eps\),
such that for every \(0<\eps<\eps_N\) there is a positive smooth
function \(r=r_\eps\) on \([0,\ell_N]\) with the following
properties. The function \(F\) is positive on \((0,\ell_N]\),
extends to an odd smooth function near zero, and satisfies
\(F'(0)=1\); the function \(r\) is constant near zero. The metric
\eqref{eq:north} therefore extends smoothly over \(D^7\times S^3\).
Its star action is free and isometric, and every star-horizontal
two-plane has strictly positive sectional curvature. The quotient
\((D_N,g_N)\) is a smooth marked seven-disk with
\(\Ksec_{g_N}>0\), including at its center and boundary.

The endpoint values satisfy
\begin{equation}\label{eq:northendpoints}
 F(\ell_N)=F_a,\qquad F'(\ell_N)>0,\qquad
 r(\ell_N)=r_a,\qquad \frac{r'(\ell_N)}{r_a}=q_s.
\end{equation}
Thus the source boundary metric is
\(F_a^2h_{S^6}+r_a^2h_{S^3}\). For a source boundary tangent
vector \(Z=X+U\) in orthonormal angular and fibre components, its
outward second fundamental form \(\mathcal B_N\) is
\begin{equation}\label{eq:northsourceform}
 \mathcal B_N(Z,Z)=\mu_N|X|^2+q_s|U|^2,
 \qquad \mu_N:=\frac{F'(\ell_N)}{F_a}>0.
\end{equation}
In particular, \(\mu_N\) is independent of \(\eps\).
\end{theorem}

We first obtain a curvature criterion for the star-horizontal
planes, then choose the two warping functions and check the center.
The resulting quotient also satisfies the explicit lower bound
\eqref{eq:northlowerbound}.

\subsection{Horizontal planes and the curvature formula}
At \(x=(p,w)\in S^6\), let
\[
K_x:\operatorname{Im}\Hh\longrightarrow T_xS^6
\]
be the infinitesimal conjugation action. Splitting \(w\) into its
imaginary and real parts, in that order, gives
\[
K_x\xi=2(\xi\times p,\xi\times\operatorname{Im}w,0),
\qquad \|K_x\|_{\rm op}\le2.
\]
The bound follows from
\(|p|^2+|\operatorname{Im}w|^2\le1\). We write \(K_x^*\) for
the adjoint with respect to \(Q\) and the unit round metric
\(h_{S^6}\).

Principal-right translation is an isometry commuting with the star
action, so it suffices to work at \(u=1\). In orthonormal
angular and fibre components a star generator is \((FK_x\xi,r\xi)\).
A star-horizontal vector is therefore of the form
\begin{equation}\label{eq:horizontalgraph}
\lambda\partial_s+X+U,\qquad
U=TX,\qquad T=-\frac F rK_x^*,\qquad \|T\|_{\rm op}\le\frac{2F}{r}.
\end{equation}
Principal-right translation transports this description isometrically
to every point of the fibre.

For a second horizontal vector \(\mu\partial_s+Y+V\), with \(V=TY\),
we have
\begin{align}
|\lambda V-\mu U|^2
  &\le\frac{4F^2}{r^2}|\lambda Y-\mu X|^2,
\label{eq:radialarea}\\
|X\otimes V-Y\otimes U|^2
  &\le\frac{8F^2}{r^2}|X\wedge Y|^2.
\label{eq:mixedarea}
\end{align}
The first inequality is the operator-norm estimate for \(T\).
To prove the second when \(X,Y\) are independent, replace the pair
by an area-preserving linear change for which they are orthogonal.
The bivector and the tensor on the left are unchanged. In the new
pair the left side is
\[
|X|^2|TY|^2+|Y|^2|TX|^2
          \le2\|T\|_{\rm op}^2|X|^2|Y|^2.
\]
For linearly dependent \(X,Y\), both sides of
\eqref{eq:mixedarea} vanish. These estimates express the vertical
area terms in terms of the projected base areas.

For arbitrary vectors as above in the doubly warped source,
the full sectional numerator is
\begin{align}
\cK={}&-\frac{F''}{F}|\lambda Y-\mu X|^2
       -\frac{r''}{r}|\lambda V-\mu U|^2
       +\frac{1-F'^2}{F^2}|X\wedge Y|^2\notag\\
 &+\frac{1-r'^2}{r^2}|U\wedge V|^2
       -\frac{F'r'}{Fr}|X\otimes V-Y\otimes U|^2.
\label{eq:warpedcurvature}
\end{align}
To obtain this formula, take orthonormal frames in the two round
factors and evaluate at a point where their intrinsic frames are
normal:
\[
\begin{gathered}
\nabla_{\partial_s}E_i=\nabla_{\partial_s}E_a=0,\qquad
\nabla_{E_i}\partial_s=(F'/F)E_i,\qquad
\nabla_{E_a}\partial_s=(r'/r)E_a,\\
\nabla_{E_i}E_j=-(F'/F)\delta_{ij}\partial_s,\qquad
\nabla_{E_a}E_b=-(r'/r)\delta_{ab}\partial_s,\qquad
\nabla_{E_i}E_a=\nabla_{E_a}E_i=0.
\end{gathered}
\]
Differentiating the connection formulas and including the intrinsic
curvature tensors of the two round factors gives a curvature operator
that is
diagonal across the radial--first-factor, radial--second-factor,
first-factor, second-factor and mixed-factor bivector blocks.
Their coefficients are exactly the five coefficients in
\eqref{eq:warpedcurvature}.

\begin{lemma}[Horizontal curvature criterion]\label{lem:northcriterion}
Suppose \(F,r>0\), \(F',r'\ge0\), and
\begin{equation}\label{eq:northcriterion}
\frac{1-F'^2}{F^2}>\frac{8FF'r'}{r^3},\qquad
-\frac{F''}{F}>\frac{4F^2(r'')_+}{r^3},\qquad r'<1.
\end{equation}
Then every star-horizontal two-plane has positive sectional curvature.
\end{lemma}
\begin{proof}
Insert \eqref{eq:radialarea} and \eqref{eq:mixedarea} into
\eqref{eq:warpedcurvature}. The possible negative radial--fibre
term is bounded using \((r'')_+=\max\{r'',0\}\).
The inequalities leave strictly positive coefficients of
\[
|\lambda Y-\mu X|^2,\qquad |X\wedge Y|^2,
\]
and a nonnegative remaining contribution from \(|U\wedge V|^2\).
The first two areas cannot both vanish on an independent horizontal
pair. Indeed, projection of the graph \eqref{eq:horizontalgraph}
to \(\Rr\partial_s\oplus T_xS^6\) is injective; their simultaneous
vanishing would make its two projected vectors, and hence the
original two vectors, dependent. This proves the lemma.
\end{proof}

\begin{remark}[The two horizontal distributions]\label{rem:horizontaldistributions}
The principal-horizontal space of \eqref{eq:north} is tangent to the
base disk. The star-horizontal space is the graph
\eqref{eq:horizontalgraph}. On this graph, \eqref{eq:radialarea}
and \eqref{eq:mixedarea} bound the radial--fibre and mixed angular
contributions by base-area terms. Lemma~\ref{lem:northcriterion}
therefore allows \(r''\) to be positive. This accommodates a radius
that is constant near the center and satisfies
\(r'(\ell_N)=r_aq_s>0\) at the boundary.
\end{remark}

\subsection{Choice of the warping functions}

The prescribed positive boundary fibre slope explains the different
roles of \(F\) and \(r\). A radius that is constant near the
center and has positive derivative at the boundary cannot be
concave throughout the disk. On star-horizontal planes, the
resulting radial loss is bounded by \(4F^2(r'')_+/r^3\), whereas
the angular loss is bounded by \(8FF'r'/r^3\). Small fibres alone
do not make the latter harmless: at the boundary,
\[
 \frac{r'(\ell_N)}{r(\ell_N)^3}
 =\frac{q_s}{r_a^2}\longrightarrow\frac{A_0F_a}{2}
 \qquad(\eps\downarrow0).
\]
We therefore choose the angular geometry before the common fibre
scale.

The profile \eqref{eq:F} provides radial curvature of order
\(\delta^{-2}\) where the fibre slope changes, and makes \(F'\)
decay rapidly as \(F/\delta\) grows. We confine the change of
fibre slope to \(\delta/4\le s\le\delta/2\). On this interval
the radial loss is of order at most \(A_0F_aC_\eta\delta\);
outside it \(r''=0\), so this loss vanishes. For the angular
coefficient, the useful ratio is
\[
 \frac{(1-F'^2)/F^2}{FF'}
 =\delta^{-3}\frac{1-e^{-x^2}}{x^3e^{-x^2/2}},
 \qquad x=F/\delta.
\]
The positive lower bound in \eqref{eq:elementary} shows why
choosing \(\delta^3A_0F_a\) small controls the angular loss on
the whole disk. The estimates below make both comparisons
quantitative. Although this profile can produce a long interval,
its length is finite and fixed before \(\eps\) is chosen.

Fix a smooth nondecreasing function \(\eta:\Rr\to[0,1]\) with
\[
\eta=0\ \text{on }(-\infty,1/4],\qquad
\eta=1\ \text{on }[1/2,\infty).
\]
Let \(C_\eta=\sup|\eta'|\). For the fixed \(A_0,F_a\), choose
\begin{equation}\label{eq:delta}
0<\delta\le F_a/2,\qquad
\delta^3\le\frac{1}{128A_0F_a(1+C_\eta)}.
\end{equation}
Define \(F\) by
\begin{equation}\label{eq:F}
F'=e^{-F^2/(2\delta^2)},\qquad F(0)=0,
\qquad
\ell_N=\int_0^{F_a}e^{z^2/(2\delta^2)}\,dz.
\end{equation}
Equivalently, \(s=\int_0^{F(s)}e^{z^2/(2\delta^2)}\,dz\).
This defines a smooth increasing solution with
\(F(\ell_N)=F_a\). Both \(F\) and the finite length \(\ell_N\)
are determined by \(F_a,\delta\).

For the values \(r_a,q_s\) in \eqref{eq:boundarydata}, put
\begin{equation}\label{eq:r}
d=r_aq_s,\qquad
r(s)=r_a-d\int_s^{\ell_N}\eta(v/\delta)\,dv.
\end{equation}
For the northern filling, require
\begin{equation}\label{eq:epsnorth}
q_s\ell_N\le\frac12,\qquad d<1.
\end{equation}
These are upper bounds satisfied for all sufficiently small
\(\eps>0\), since \(q_s=O(\eps)\) and \(d=O(\eps^{3/2})\), with
all other parameters fixed. We obtain
\begin{equation}\label{eq:rproperties}
\frac{r_a}{2}\le r\le r_a,\qquad
r'=d\eta(s/\delta)\ge0,\qquad
r''=\frac d\delta\eta'(s/\delta)\ge0.
\end{equation}
Since \(\ell_N\ge F_a\ge2\delta\), the function \(\eta(s/\delta)\)
is one near the outer boundary. Consequently
\begin{equation}\label{eq:rjet}
r(\ell_N)=r_a,\qquad \frac{r'(\ell_N)}{r_a}=q_s.
\end{equation}

We verify \eqref{eq:northcriterion} on the whole interval
\((0,\ell_N]\). First, \eqref{eq:rproperties} gives
\begin{equation}\label{eq:density}
\frac d{r^3}\le\frac{8q_s}{r_a^2}
       =4A_0F_a e^{-\eps A_0|\cos a|}\le4A_0F_a.
\end{equation}
Differentiating \eqref{eq:F} gives
\[
F''=-\frac F{\delta^2}(F')^2.
\]
Thus
\begin{equation}\label{eq:Fcurvature}
-\frac{F''}{F}=\delta^{-2}e^{-F^2/\delta^2},\qquad
\frac{1-F'^2}{F^2}=\frac{1-e^{-F^2/\delta^2}}{F^2}.
\end{equation}
The support of \(r''\) is contained in \(s\le\delta/2\).
Since \(F'\le1\), we have \(F(s)\le s\), and hence there
\[
-\frac{F''}{F}\ge e^{-1/4}\delta^{-2},
\qquad
\frac{4F^2r''}{r^3}\le4A_0F_aC_\eta\delta
 \le\frac{C_\eta}{32(1+C_\eta)}\delta^{-2}.
\]
This proves the radial inequality. Off the support of \(r''\)
its right side vanishes and its left side remains strictly positive.

For the angular inequality, for every \(x>0\) we have
\begin{align}
\frac{1-e^{-x^2}}{x^3e^{-x^2/2}}
 &=\frac{2\sinh(x^2/2)}{x^3}
 \ge\frac1x+\frac{x^3}{24}
 \ge\frac4{3\,8^{1/4}}>\frac12.
\label{eq:elementary}
\end{align}
The first bound follows from the first two nonzero terms in the
power series of \(\sinh\), whose subsequent terms are positive.
The minimum of \(x^{-1}+x^3/24\) occurs at \(x^4=8\), giving
the second bound. Substitution of \(x=F/\delta\) gives
\[
\frac{1-F'^2}{F^2}>
       \frac{FF'}{2\delta^3}>
       32A_0F_aFF'\ge\frac{8FF'r'}{r^3},
\]
where we used \eqref{eq:delta}, \eqref{eq:density}, and \(r'\le d\).
Finally, \(r'\le d<1\), so all three inequalities in
\eqref{eq:northcriterion} hold.

\subsection{The center and the quotient}
The solution \(F\) extends to an odd smooth function near zero,
with \(F'(0)=1\). Indeed, it is the local inverse of the odd smooth function \(\int_0^F e^{z^2/(2\delta^2)}\,dz\), whose
derivative at zero is one. These are the polar smoothness conditions
for \(ds^2+F^2h_{S^6}\). More explicitly,
\[
 F(s)=s-\frac{s^3}{6\delta^2}+O(s^5).
\]
The function \(r\) is constant near zero
by the choice of \(\eta\), so the entire source metric is smooth
there.

Both expressions in \eqref{eq:Fcurvature} tend to
\(\delta^{-2}>0\) at the center. In a neighborhood of the center
the source is a product of this positively curved base with a
round \(S^3\) of constant radius. Its sectional numerator on
star-horizontal planes is positive, since projection of such a
plane to the base is injective. At the center itself the star orbit
is entirely in the fibre and the horizontal space is the whole
base tangent space. Thus positivity also holds at that point.

Lemma~\ref{lem:northcriterion} now proves positivity on all
star-horizontal planes throughout the source, including the
boundary. By \eqref{eq:ONeill}, the star quotient \(D_N\) is a
smooth positively curved compact disk. The estimates use the bound
\(\|K_x\|_{\rm op}\le2\), which holds on all of \(S^6\), including
points where the conjugation orbit has smaller dimension. More explicitly, at \(u=1\) the Gram matrix of the star
infinitesimal generators is
\[
 F^2K_x^*K_x+r^2\Id_{\operatorname{Im}\Hh}
 \ge r^2\Id_{\operatorname{Im}\Hh}>0.
\]
Thus a change in the rank of \(K_x\) does not make the star orbit
singular. The horizontal graph formula uses \(K_x^*\), not an
inverse of \(K_x\), and remains valid at every orbit type.

\subsection{A quantitative lower bound on the quotient}
We record a lower bound in terms of the chosen parameters. This also
quantifies the passage from the horizontal curvature estimate to the
quotient metric. Put
\[
 T_*:=F_a^2/\delta^2\ge4,\qquad
 p_*:=\delta^{-2}e^{-T_*},\qquad
 q_*:=\frac{1-e^{-T_*}}{2F_a^2}.
\]
Let \(P_{\mathrm{rad}},P_{\mathrm{ang}}\) be the two remaining
base-area coefficients after applying \eqref{eq:radialarea} and
\eqref{eq:mixedarea}:
\[
 P_{\mathrm{rad}}=-F''/F-4F^2(r'')_+/r^3,
 \qquad P_{\mathrm{ang}}=(1-F'^2)/F^2-8FF'r'/r^3.
\]
On the support of \(r''\), the radial estimates above give
\(P_{\mathrm{rad}}>15/(32\delta^2)>p_*\); off this support,
\(P_{\mathrm{rad}}=\delta^{-2}e^{-F^2/\delta^2}\ge p_*\).
For the angular term, \eqref{eq:delta}, \eqref{eq:density} and
\eqref{eq:elementary} imply
\[
 \frac{8FF'r'}{r^3}
 \le\frac{FF'}{4(1+C_\eta)\delta^3}
 <\frac12\frac{1-F'^2}{F^2}.
\]
The function \((1-e^{-y})/y\) is decreasing for \(y>0\), since
\((y+1)e^{-y}<1\). It follows that
\[
 P_{\mathrm{ang}}\ge\frac12\frac{1-e^{-F^2/\delta^2}}{F^2}\ge q_*.
\]
All expressions at the center are interpreted by their smooth limits.

The horizontal graph injection in \eqref{eq:horizontalgraph} has
squared operator norm at most
\[
 H_*:=1+16F_a^2/r_a^2,
\]
since \(F\le F_a\) and \(r\ge r_a/2\). Its induced map on second
exterior powers has squared operator norm at most \(H_*^2\).
Consequently, for any unit star-horizontal bivector, the sum of its
two base-area squares is at least \(H_*^{-2}\). The vertical-area term in \eqref{eq:warpedcurvature} is nonnegative.
O'Neill's formula therefore gives
\begin{equation}\label{eq:northlowerbound}
 \Ksec_{g_N}\ge c_N:=
 \left(1+\frac{16F_a^2}{r_a^2}\right)^{-2}
 \min\left\{\frac{e^{-F_a^2/\delta^2}}{\delta^2},
       \frac{1-e^{-F_a^2/\delta^2}}{2F_a^2}\right\}>0.
\end{equation}
For fixed parameters this bound is positive, and it extends to the
center by continuity.

\begin{proof}[Proof of Theorem~\ref{thm:northfilling}]
Choose \(\eta,\delta\) as above and define \(F,\ell_N\) by
\eqref{eq:F}. These choices depend on \(A_0,F_a\), not on
\(\eps\). Choose \(\eps_N>0\) so that \eqref{eq:epsnorth}
holds for every \(0<\eps<\eps_N\), and define \(r\) by
\eqref{eq:r}. Equations~\eqref{eq:rproperties} and
\eqref{eq:rjet} give its positivity, its constant value near the
center, and its prescribed endpoint value and derivative.
The preceding estimates verify Lemma~\ref{lem:northcriterion} on
\((0,\ell_N]\); the center calculation proves the smooth extension
and positivity there. The free star quotient is a marked disk by
Section~\ref{sec:carrier}, and O'Neill's formula gives
\eqref{eq:northlowerbound}. Finally, \eqref{eq:F} and
\eqref{eq:rjet} yield \eqref{eq:northendpoints}. Differentiating
\eqref{eq:north} in the outward direction \(\partial_s\) gives
\eqref{eq:northsourceform}.
\end{proof}

\section{Boundary compatibility and gluing}\label{sec:gluing}

The two filling theorems give metrics with the same source boundary
radii and fibre slope. We now compare their quotient boundaries under
the prescribed marked map. Denote the induced boundary metrics by
\(h_N,h_S\) and the outward second fundamental forms by \(B_N,B_S\).

\Needspace{12\baselineskip}
\begin{proposition}[Boundary compatibility]\label{prop:boundarycompatibility}
Fix \(n\in\mathbb Z\) and the data in
Theorems~\ref{thm:southfilling} and \ref{thm:northfilling}, using
the same \(a,A_0\). For any
\(0<\eps<\min\{\eps_S,\eps_N\}\), let \((D_S,g_S)\) and
\((D_N,g_N)\) be the resulting fillings. The marked map
\(\phi_n:\partial D_N\to\partial D_S\), represented by
\(\sigma^n\) in the extending disk coordinates, is a boundary
isometry:
\[
 h_N=\phi_n^*h_S=:h.
\]
With \(\mu_N=F'(\ell_N)/F_a>0\) and \(\mu_S=\cot a<0\),
\begin{equation}\label{eq:boundarymargin}
 B_N+\phi_n^*B_S\ge c_B h,
 \qquad c_B:=\frac{r_a^2}{r_a^2+4F_a^2}(\mu_N-\mu_S)>0.
\end{equation}
\end{proposition}

We establish the metric equality in the common north gauge and then
compute the second fundamental forms on the same horizontal lifts.

\subsection{The common boundary metric}
In the north gauge the two source boundary metrics are both
\begin{equation}\label{eq:sourceboundary}
F_a^2h_{S^6}+r_a^2h_{S^3}.
\end{equation}
The star actions agree, so the induced quotient boundary metrics
are exactly equal under the identity in that gauge.
Write \(\varpi_\partial(x,u)=u^{-1}xu\) for the boundary
quotient map in the common angular marking. At \(u=1\), its
differential in orthonormal components is
\[
d\varpi_\partial(X,U)=F_a^{-1}X-r_a^{-1}K_xU.
\]
The source inner product in these components is the identity.
Hence, for \(L=(F_a^{-1}\Id,-r_a^{-1}K_x)\), the quotient
cometric is \(LL^*\), namely
\begin{equation}\label{eq:cometric}
h^{-1}=F_a^{-2}h_{S^6}^{-1}+r_a^{-2}K_xK_x^*.
\end{equation}
By Remark~\ref{rem:boundarycoordinates}, the identity in this
trivialization is represented by \(\sigma^n\) in the extending disk
coordinates.

\subsection{The second fundamental forms}
Denote the boundary isometry by
\(\phi_n:\partial D_N\to\partial D_S\). All comparisons below are
on \(\partial D_N\), with the southern tensors pulled back by
\(\phi_n\). Let \(Y\ne0\) be a tangent vector to this boundary, and
write its common horizontal lift as \(X+U\). The outward unit
normal is \(\partial_s\) on the north and \(-\partial_t\) on
the south. These normals are horizontal for the star submersions.
The horizontal component of the covariant derivative of a basic
horizontal normal projects to the quotient covariant derivative.
Taking its inner product with the horizontal lift of a boundary
vector identifies the quotient second fundamental form with the
source form restricted to horizontal lifts. Since \(A_N=0\) on the
boundary collar, the source forms follow by differentiating the two
warping factors. With
\[
\mu_N=\frac{F'(\ell_N)}{F_a}>0,\qquad \mu_S=\cot a<0,
\]
equation \eqref{eq:rjet} gives
\begin{align}
B_N(Y,Y)&=\mu_N|X|^2+q_s|U|^2,\notag\\
(\phi_n^*B_S)(Y,Y)&=-\mu_S|X|^2-q_s|U|^2.
\label{eq:forms}
\end{align}
Their sum is therefore
\begin{equation}\label{eq:sum}
(B_N+\phi_n^*B_S)(Y,Y)=(\mu_N-\mu_S)|X|^2>0.
\end{equation}
By \eqref{eq:horizontalgraph}, a nonzero horizontal vector has
\(X\ne0\). Together with \(\mu_N>0>\mu_S\), this proves strict
positive definiteness.

To express the eigenvalues relative to \(h\), let \(b_j\) be the
eigenvalues of \(K_xK_x^*\), so \(0\le b_j\le4\).
For \(X\) in a corresponding eigendirection,
\[
|Y|_h^2=|X|^2+|U|^2
       =\left(1+\frac{F_a^2b_j}{r_a^2}\right)|X|^2.
\]
The eigenvalues of the shape sum relative to \(h\) are thus
\begin{equation}\label{eq:sum-eigenvalues}
\frac{r_a^2}{r_a^2+F_a^2b_j}(\mu_N-\mu_S)>0.
\end{equation}
Since \(b_j\le4\), these eigenvalues are bounded below by the
constant \(c_B\) in \eqref{eq:boundarymargin}. At
\(x=(i,0)\), the map \(K_xK_x^*\) has eigenvalue \(4\), so
\(c_B\) is attained. On \(\ker K_x^*\), the eigenvalue of the
shape sum is \(\mu_N-\mu_S=\mu_N+|\cot a|\).

\begin{proof}[Proof of Proposition~\ref{prop:boundarycompatibility}]
Theorems~\ref{thm:southfilling} and \ref{thm:northfilling} give the
common source boundary metric \eqref{eq:sourceboundary}. Taking the
quotient by the common star action gives the boundary isometry above,
whose expression in the extending coordinates is \(\sigma^n\) by
\eqref{eq:clutching}. The horizontal restriction of the source forms
\eqref{eq:southsourceform} and \eqref{eq:northsourceform} gives
\eqref{eq:forms}. Their sum and its eigenvalues are
\eqref{eq:sum} and \eqref{eq:sum-eigenvalues}, proving
\eqref{eq:boundarymargin}.
\end{proof}

\subsection{The compatible filling theorem}
\begin{proof}[Proof of Theorem~\ref{thm:fillings}]
Fix \(n\in\mathbb Z\). Choose \(a,b_0,\chi\) and the southern
connection, which determines \(M_0,M_1\). Choose \(\Lambda,A_0\)
as in Theorem~\ref{thm:southfilling}; this theorem supplies a
family of southern fillings for \(0<\eps<\eps_S\).
With this same \(a,A_0\), Theorem~\ref{thm:northfilling}
supplies \(F,\ell_N\) and a family of northern fillings for
\(0<\eps<\eps_N\). The northern angular profile is fixed before
\(\eps\), so both thresholds are determined before making the
common choice
\[
 0<\eps<\min\{\eps_S,\eps_N\}.
\]
In particular, \(\eps\) satisfies
\eqref{eq:smallA}--\eqref{eq:smallC} and \eqref{eq:epsnorth}
for the fixed data.

Theorems~\ref{thm:southfilling} and \ref{thm:northfilling} give
smooth positive metrics on the entire marked disks for this same
\(\eps\). Proposition~\ref{prop:boundarycompatibility} gives
\(h_N=\phi_n^*h_S\) and \(B_N+\phi_n^*B_S>0\) under their
prescribed marked identification. These are precisely the conclusions
of Theorem~\ref{thm:fillings}.
\end{proof}

\subsection{Positive gluing and the main theorem}
The \(k=1\) case of the Reiser--Wraith gluing theorem
\cite[Theorem~A(i)]{RW} concerns positive sectional curvature.
It assumes that the boundary metrics are isometric and that the sum
of the second fundamental forms is positive semidefinite. We use the
following strict case. With inward normal \(N=-n_{\rm out}\), our
convention is \(B(Y,Z)=-\langle\nabla_YN,Z\rangle\), which agrees
with the convention in~\cite{RW}.

\Needspace{12\baselineskip}
\begin{lemma}[Positive gluing]\label{lem:gluing}\label{thm:RW}
Let \((M_N,g_N)\) and \((M_S,g_S)\) be smooth Riemannian manifolds
of the same dimension at least two, with compact boundaries and
strictly positive sectional curvature, including at the boundaries.
Let \(h_N,h_S\) and \(B_N,B_S\) be their induced boundary metrics
and outward second fundamental forms, with convention
\eqref{eq:boundaryconvention}. Suppose a boundary diffeomorphism
\(\phi:\partial M_N\to\partial M_S\) satisfies
\[
 h_N=\phi^*h_S=:h,\qquad B_N+\phi^*B_S\ge b_*h
 \quad\text{for some }b_*>0.
\]
Then the glued smooth manifold \(M_N\cup_\phi M_S\) admits a smooth
metric of strictly positive sectional curvature. For every
neighborhood of the identified boundary, the metric may be chosen
to agree with \(g_N\) and \(g_S\) outside that neighborhood.
\end{lemma}
\begin{proof}
See Appendix~\ref{app:gluing}.
\end{proof}

For the quotient disks, the signed normal coordinate is
\[
 z=s-\ell_N\le0\quad\text{on }D_N,\qquad
 z=t-a\ge0\quad\text{on }D_S.
\]
The radial fields are unit and star-horizontal, and the star orbits
are tangent to the radial slices. After pulling back the southern
collar by \(\phi_n\), the quotient metrics therefore have the form
\(dz^2+h_-(z)\) and \(dz^2+h_+(z)\), with
\[
 h_-(0)=h_+(0)=h,\qquad
 h_-'(0)-h_+'(0)=2(B_N+\phi_n^*B_S)\ge2c_Bh.
\]
The fibre radii have the common value \(r_a\) and common derivative
\(r_aq_s\) at \(z=0\). The angular derivatives give the positive
jump displayed above. Thus the two metrics define a continuous
metric with a corner along the identified boundary.

Fix the common \(\eps\) as in the proof of
Theorem~\ref{thm:fillings}. The curvature bounds \(c_S,c_N\) and
the boundary constant \(c_B\) are then positive fixed numbers.
The interpolation width \(\tau\) is chosen for these metrics,
and the smoothing width \(\nu\) is chosen after \(\tau\).

\begin{proof}[Proof of Theorem~\ref{thm:main}]
Choose an orientation on the given homotopy seven-sphere \(\Sigma\).
By the coordinate identification in Lemma~\ref{lem:topology} and
\cite[Theorem~1]{DPR}, with the orientation comparison in
Section~\ref{sec:carrier}, its oriented smooth type is represented by
\(\Sigma_n\) for some \(n\in\{0,\ldots,27\}\).
Theorem~\ref{thm:fillings} provides smooth positive metrics on the
two marked disks defining \(\Sigma_n\), with isometric boundaries
and a positive definite sum of outward second fundamental forms.
Their boundaries are compact, so Lemma~\ref{lem:gluing} applies
and gives a smooth positively curved metric on
\(D_N\cup_{\phi_n}D_S=\Sigma_n\). The quantitative version in~\cite[Corollary~B, \(k=1\)]{RW}
applies with lower threshold \(\tfrac12\min\{c_S,c_N\}\), since
both disk metrics have strictly larger sectional curvature. Thus the
smoothed metric may be chosen to satisfy
\[
 \Ksec_g>\tfrac12\min\{c_S,c_N\}>0.
\]
Pulling this metric back to \(\Sigma\) proves the theorem, with
\(c_\Sigma=\tfrac12\min\{c_S,c_N\}\).
\end{proof}

\Needspace{10\baselineskip}
\appendix
\section{The principal bundle \texorpdfstring{\(P_n\)}{Pn}}\label{app:principal}

We give the quotient-space construction underlying the clutching
description in Section~\ref{sec:carrier}, using the notation fixed
there.

For each \(n\in\mathbb Z\), the principal right
\(S^3\)-bundle
\[
 \pi_n:P_n\longrightarrow S^7
\]
has structure group \(S^3\), the group of unit quaternions under
multiplication. Its total space is
\[
 P_n:=\bigl((\mathcal U_N\times S^3)\sqcup
                 (\mathcal U_S\times S^3)\bigr)/{\sim_n},
\]
where \(\sqcup\) denotes disjoint union. Write \((y,u)_\alpha\)
for a point in the \(\alpha\)-th product, with \(\alpha\in\{N,S\}\).
The equivalence relation \(\sim_n\) is generated by identifying
\((y,u_N)_N\) with \((y,u_S)_S\) whenever the common base point is
\(y=(\cos t,\sin t\,x)\in\mathcal U_N\cap\mathcal U_S\) and
\begin{equation}\label{eq:principal-appendix}
u_S=g_n(x)u_N.
\end{equation}
Equip \(P_n\) with the quotient topology and denote these classes
by \([y,u]_\alpha\). Since the identifications preserve the base
point, the product projections induce a well-defined continuous map
\[
 \pi_n:P_n\longrightarrow S^7,\qquad
 \pi_n([y,u]_\alpha)=y.
\]

The canonical maps
\[
 \iota_\alpha:\mathcal U_\alpha\times S^3
       \longrightarrow\pi_n^{-1}(\mathcal U_\alpha),\qquad
 \iota_\alpha(y,u)=[y,u]_\alpha,
 \quad \alpha\in\{N,S\},
\]
are homeomorphisms onto open subsets covering \(P_n\).
The quotient is Hausdorff: points with different projections are
separated using the base, and points with the same projection are
separated in a common product trivialization. It is second countable
because these two open subsets are second countable. On the overlap,
\[
 (\iota_S^{-1}\circ\iota_N)(y,u)
       =(y,g_n(x)u),\qquad y=(\cos t,\sin t\,x).
\]
This transition is a smooth diffeomorphism, with inverse
\((y,u)\mapsto(y,g_n(x)^{-1}u)\). We give \(P_n\) the unique smooth
structure for which both \(\iota_N\) and \(\iota_S\) are
diffeomorphisms. Thus \(P_n\) is a smooth ten-dimensional manifold
without boundary, and \(\pi_n\) is a smooth locally trivial bundle
with fibre \(S^3\).

Define the \emph{principal right action} by
\[
 [y,u]_\alpha\cdot h:=[y,uh]_\alpha,\qquad h\in S^3.
\]
The identity \(g_n(x)(uh)=(g_n(x)u)h\) shows that this definition
is independent of the representative. In each trivialization it is
the standard smooth right multiplication on \(S^3\); hence it
preserves \(\pi_n\) and is free and transitive on every fibre.
Consequently, \(\pi_n:P_n\to S^7\), with these trivializations and
this action, is a smooth principal right \(S^3\)-bundle.

\section{Proof of the positive gluing lemma}\label{app:gluing}

We prove Lemma~\ref{lem:gluing} by the cubic interpolation of
Reiser--Wraith~\cite[Section~2]{RW}, followed by smoothing in normal
collars. Throughout the proof, the two metrics are fixed. The
interpolation width is denoted by \(\tau>0\).

\subsection{Signed collars and the boundary inequality}
If the common boundary is empty, the disjoint union of the two
original metrics already proves the lemma. Assume henceforth
that it is nonempty. Set \(X=\partial M_N\), and use \(\phi\) to pull the southern
collar back to \(X\). In signed normal coordinates, increasing
from the northern side to the southern side, the metrics take
the form
\begin{equation}\label{eq:glue-collars}
 g_N=dz^2+h_-(z)\quad(-r_0\le z\le0),\qquad
 g_S=dz^2+h_+(z)\quad(0\le z\le r_0),
\end{equation}
where \(h_-(0)=h_+(0)=h\). Here and below the notation \(g_S\)
on the collar includes this pullback. The families \(h_\pm\) are
smooth up to \(z=0\). Compactness of \(X\) allows a common
\(r_0>0\); the collars may be chosen inside any prescribed open
neighborhood of the identified boundary. These collars define the
usual smooth structure on the gluing by \(\phi\).

The outward normals are \(+\partial_z\) on the northern side
and \(-\partial_z\) on the southern side. Consequently
\begin{equation}\label{eq:glue-jump}
 P_0:=h_-'(0),\qquad Q_0:=h_+'(0),\qquad
 \Delta:=P_0-Q_0=2(B_N+\phi^*B_S)\ge2b_*h.
\end{equation}
Primes denote \(z\)-derivatives. All tensor estimates below are
uniform on \(X\), measured using \(h\) and a finite collection of
fixed coordinate charts. An estimate in \(C_X^j\) includes
all tangential derivatives of order at most \(j\), but not
normal derivatives. The latter will be stated separately.

\subsection{Cubic interpolation and its derivatives}
For \(0<\tau<r_0/4\), put
\[
 a_\tau=h_-(-\tau),\quad b_\tau=h_+(\tau),\quad
 P_\tau=h_-'(-\tau),\quad Q_\tau=h_+'(\tau),\quad
 D_\tau=\frac{b_\tau-a_\tau}{2\tau}.
\]
On \(-\tau\le z\le\tau\), define
\begin{align}
 H_z={}&\frac{z+\tau}{2\tau}b_\tau
        -\frac{z-\tau}{2\tau}a_\tau\notag\\
 &+\frac{(z-\tau)^2(z+\tau)}{4\tau^2}(P_\tau-D_\tau)
  +\frac{(z+\tau)^2(z-\tau)}{4\tau^2}(Q_\tau-D_\tau).
 \label{eq:glue-hermite}
\end{align}
This is a polynomial with values in
\(C^\infty(X,\operatorname{Sym}^2T^*X)\). Its endpoint data are
\begin{equation}\label{eq:glue-endpoint-jets}
 H_{-\tau}=a_\tau,\quad H_\tau=b_\tau,\quad
 H'_{-\tau}=P_\tau,\quad H'_\tau=Q_\tau.
\end{equation}
Differentiating the polynomial gives
\begin{align}
 H_z'={}&D_\tau+
 \frac{3z^2-\tau^2}{4\tau^2}(P_\tau+Q_\tau-2D_\tau)
 +\frac{z}{2\tau}(Q_\tau-P_\tau),
 \label{eq:glue-hermite-first}\\
 H_z''={}&\frac{3z}{2\tau^2}(P_\tau+Q_\tau-2D_\tau)
 +\frac{Q_\tau-P_\tau}{2\tau}.
 \label{eq:glue-hermite-second}
\end{align}
Taylor expansion of the fixed smooth families yields, in every
fixed tangential \(C^j\)-norm,
\begin{gather*}
 a_\tau=h-\tau P_0+O(\tau^2),\qquad
 b_\tau=h+\tau Q_0+O(\tau^2),\\
 P_\tau=P_0+O(\tau),\qquad Q_\tau=Q_0+O(\tau),\qquad
 D_\tau=\tfrac12(P_0+Q_0)+O(\tau).
\end{gather*}
In particular,
\(P_\tau+Q_\tau-2D_\tau=O(\tau)\) and
\(Q_\tau-P_\tau=-\Delta+O(\tau)\). Substitution into
\eqref{eq:glue-hermite}--\eqref{eq:glue-hermite-second}, with
\(\lambda=(z+\tau)/(2\tau)\), gives the uniform estimates
\begin{align}
 H_z&=h+O(\tau)&&\text{in }C_X^2,
 \label{eq:glue-estimate-metric}\\
 H_z'&=(1-\lambda)P_0+\lambda Q_0+O(\tau)
       &&\text{in }C_X^1,
 \label{eq:glue-estimate-first}\\
 H_z''&=-\frac{\Delta}{2\tau}+O(1)
       &&\text{in }C_X^0.
 \label{eq:glue-estimate-second}
\end{align}
For example, the first term on the right side of
\eqref{eq:glue-hermite-second} is \(O(1)\), since
\(|z|\le\tau\) and its bracket is \(O(\tau)\). To obtain
\eqref{eq:glue-estimate-first}, the \(O(1)\) part is
\((P_0+Q_0)/2+z(Q_0-P_0)/(2\tau)\), which is the stated
linear interpolation. The same calculations commute with
tangential differentiation because all scalar coefficients
depend only on \(z\).

After decreasing \(\tau\), \eqref{eq:glue-estimate-metric}
ensures \(\tfrac12h\le H_z\le2h\). Thus
\(G_\tau=dz^2+H_z\) is a Riemannian metric on the interpolation
strip. Extend it by \(g_N\) and \(g_S\) outside this strip.
Equation~\eqref{eq:glue-endpoint-jets}, as an identity of smooth
tensors on \(X\), matches both values and all first derivatives
at the two interfaces \(z=\pm\tau\). The extended metric is
therefore \(C^1\) and piecewise smooth. For a fixed \(\tau\),
it is locally \(C^{1,1}\), with bounded one-sided second
coordinate derivatives. We next estimate the classical curvature
on its smooth pieces and their one-sided limits.

\subsection{All curvature blocks in a normal collar}
Consider a smooth metric \(G=dz^2+H_z\). Let indices \(i,j,k,l\)
be tangent to \(X\), and let \(D^H\) be the Levi--Civita
connection of the slice metric \(H_z\). The Christoffel symbols
in the product coordinates are
\begin{equation}\label{eq:glue-christoffel}
 \Gamma^z_{ij}=-\tfrac12H'_{ij},\qquad
 \Gamma^i_{zj}=\tfrac12H^{ik}H'_{kj},\qquad
 \Gamma^i_{jk}=\Gamma^i_{jk}(H_z),\qquad
 \Gamma^A_{zz}=\Gamma^z_{zi}=0.
\end{equation}
Here a capital index can be tangential or normal. Write
\(R_{ABCD}=G(R(\partial_A,\partial_B)\partial_C,\partial_D)\).
Substituting
\eqref{eq:glue-christoffel} into the curvature definition gives
\begin{align}
 R_{ijkl}&=R^{H_z}_{ijkl}
       +\tfrac14(H'_{ik}H'_{jl}-H'_{il}H'_{jk}),
       \label{eq:glue-gauss}\\
 R_{izzj}&=-\tfrac12H''_{ij}
       +\tfrac14H'_{ik}H^{kl}H'_{lj},
       \label{eq:glue-radial}\\
 R_{ijkz}&=\tfrac12\bigl((D^H_jH')_{ik}-(D^H_iH')_{jk}\bigr).
       \label{eq:glue-codazzi}
\end{align}
For \eqref{eq:glue-gauss}, the normal Christoffel products
add the displayed quadratic term to the intrinsic curvature of
the slice. For \eqref{eq:glue-radial}, differentiating
\(H^{ik}H'_{kj}/2\) and using
\((H^{-1})'=-H^{-1}H'H^{-1}\) gives the stated expression
after lowering the remaining index. For
\eqref{eq:glue-codazzi}, the tangential derivatives of
\(-H'_{jk}/2\), together with the slice Christoffel terms,
combine into the two covariant derivatives shown. These three
blocks and the curvature symmetries describe the entire tensor:
a nonzero component has zero, one, or two normal indices, with
at most one normal index in either skew pair.

For a nonzero tangential vector \(u\), the radial formula reads
\begin{equation}\label{eq:glue-radial-sectional}
 \Ksec_G(u,\partial_z)=
 \frac{-\tfrac12H_z''(u,u)
  +\tfrac14H_z'(u,\cdot)H_z^{-1}H_z'(\cdot,u)}{H_z(u,u)}.
\end{equation}
Here \(\Ksec_G(u,v)\) means the sectional curvature of the
plane spanned by \(u,v\).

\subsection{Tangential planes and the determinant identity}
First suppose that the manifold dimension is at least three.
For a two-plane \(\Pi\subset T_xX\), choose an \(h\)-orthonormal
basis \(u,v\), extended with constant components in the product
direction. If \(S\) is a symmetric two-tensor, write
\(\det_\Pi S=S(u,u)S(v,v)-S(u,v)^2\). The Gauss equation gives
\begin{equation}\label{eq:glue-tangent-sectional}
 \Ksec_{G_\tau}(\Pi)
 =\Ksec_{H_z}(\Pi)
       -\frac{\det_\Pi H_z'}{4\det_\Pi H_z}.
\end{equation}
Since \(\det_\Pi h=1\),
\eqref{eq:glue-estimate-metric} implies
\(\det_\Pi H_z=1+O(\tau)\) and
\(\Ksec_{H_z}(\Pi)=\Ksec_h(\Pi)+O(\tau)\), uniformly over
the compact bundle of these planes. The normalization of the
basis makes the denominator uniformly bounded away from zero.

For two symmetric \(2\times2\) matrices \(P,Q\), direct
expansion of the determinant yields
\begin{equation}\label{eq:glue-determinant}
 \det((1-\lambda)P+\lambda Q)
 =(1-\lambda)\det P+\lambda\det Q
       -\lambda(1-\lambda)\det(P-Q).
\end{equation}
Apply this identity to the restrictions of \(P_0,Q_0\) to
\(\Pi\). Their difference is the restriction of \(\Delta\),
which is positive definite by \eqref{eq:glue-jump}, so
\(\det_\Pi\Delta>0\). Equations
\eqref{eq:glue-estimate-first} and
\eqref{eq:glue-tangent-sectional} give
\begin{align}
 \Ksec_{G_\tau}(\Pi)
 ={}&(1-\lambda)\Ksec_{g_N}(\Pi)
        +\lambda\Ksec_{g_S}(\Pi)\notag\\
 &+\frac{\lambda(1-\lambda)}4\det_\Pi\Delta+O(\tau).
 \label{eq:glue-tangent-bound}
\end{align}
The curvatures on the right are evaluated at the boundary,
with the southern plane transported by \(\phi\). Indeed,
their Gauss formulas are
\(\Ksec_h(\Pi)-\tfrac14\det_\Pi P_0\) and
\(\Ksec_h(\Pi)-\tfrac14\det_\Pi Q_0\), respectively.
The sign of the southern outward normal does not change this
determinant.

Both boundary curvatures are strictly positive. Compactness
therefore gives a common lower bound \(\kappa_\partial>0\)
over all tangential two-planes. Decreasing \(\tau\) once more
in \eqref{eq:glue-tangent-bound} yields
\begin{equation}\label{eq:glue-tangent-reserve}
 \Ksec_{G_\tau}(\Pi)\ge
 a_0:=\kappa_\partial/2>0
 \quad(\Pi\subset T_xX).
\end{equation}
The ambient boundary curvatures and the positive definite tensor
\(\Delta\) thus give the tangential lower bound.

\subsection{Radial and arbitrary two-planes}
For an \(H_z\)-unit tangential vector \(u\),
\eqref{eq:glue-estimate-second} and
\eqref{eq:glue-radial-sectional} imply
\[
 \Ksec_{G_\tau}(u,\partial_z)
 =\frac{\Delta(u,u)}{4\tau}+O(1).
\]
The error is uniform: \(H_z'\), \(H_z^{-1}\), and the
remainder in \eqref{eq:glue-estimate-second} are bounded
independently of sufficiently small \(\tau\). Moreover,
\(H_z\le2h\) and \(H_z(u,u)=1\) give \(h(u,u)\ge1/2\),
so \(\Delta(u,u)\ge2b_*h(u,u)\ge b_*\). There is a constant
\(C_0\ge0\)
such that
\begin{equation}\label{eq:glue-radial-reserve}
 \Ksec_{G_\tau}(u,\partial_z)
       \ge\frac{\beta_0}{\tau}-C_0,
 \qquad \beta_0:=b_*/4>0.
\end{equation}
Also, \eqref{eq:glue-estimate-metric},
\eqref{eq:glue-estimate-first}, and
\eqref{eq:glue-codazzi} give a constant \(M\ge0\),
independent of sufficiently small \(\tau\), such that
\begin{equation}\label{eq:glue-cross-reserve}
 \bigl|G_\tau(R(u,w)w,\partial_z)\bigr|\le M
\end{equation}
whenever \(u,w\) are \(H_z\)-orthonormal. In fact, this block
uses only \(H_z'\), its first tangential derivatives, and the
slice connection; all are uniformly bounded by the stated
\(C_X^1\) and \(C_X^2\) estimates. It contains no \(H_z''\).

Every two-plane in \(\Rr\partial_z\oplus T_xX\) has a
\(G_\tau\)-orthonormal basis of the form
\[
 w,\qquad \cos\theta\,u+\sin\theta\,\partial_z,
\]
with \(u,w\in T_xX\) orthonormal, when the manifold dimension
is at least three. To see this, choose a unit vector in its
intersection with \(T_xX\), and orthogonally project the
other basis vector onto \(T_xX\). If that projection vanishes,
choose any tangential unit vector orthogonal to \(w\). For a
tangential plane take \(\theta=0\). Expansion of the sectional
numerator, followed by \eqref{eq:glue-tangent-reserve}--
\eqref{eq:glue-cross-reserve}, gives
\begin{align}
 \Ksec_{G_\tau}(\Pi)
 &\ge a_0\cos^2\theta
      +\left(\frac{\beta_0}{\tau}-C_0\right)\sin^2\theta
      -2M|\sin\theta\cos\theta|\notag\\
 &\ge\frac{a_0}{2}\cos^2\theta
      +\left(\frac{\beta_0}{\tau}-C_0-\frac{2M^2}{a_0}\right)
        \sin^2\theta.
 \label{eq:glue-all-planes}
\end{align}
Choose \(\tau\) small enough that the coefficient in parentheses
in the last line is at least \(a_0/2\). Then every two-plane
has sectional curvature at least \(a_0/2\) on the interpolation
strip. This bound is uniform in the angle; it includes planes
whose angle depends on \(\tau\).

If the manifold dimension is two, the slice has dimension one
and there are no tangential two-planes. The unique two-plane at
each point is spanned by \(\partial_z\) and a slice unit vector.
Its curvature is positive for small \(\tau\) directly by
\eqref{eq:glue-radial-reserve}. Thus the interpolation conclusion
holds in that dimension as well.

Fix such a \(\tau\) from now on. On the compact original
subcollars the curvature has a positive minimum, and on the
interpolation strip the preceding estimates give a positive
minimum. Hence, on a compact collar containing both interfaces,
there is a number \(c_*>0\) bounding below the sectional
curvatures of all smooth pieces and their one-sided limits.

\subsection{A common tensor-valued mollifier}
We describe the smoothing near one interface, translated to
\(z=0\). Write the fixed piecewise smooth metric as
\(G=dz^2+k_z\). Choose \(\ell>0\) so small that
\(X\times[-3\ell,3\ell]\) contains no other interface, is
contained in the chosen collar, and is disjoint from the strip
used at the other interface. All constants in the rest of the
proof may depend on the fixed \(\tau\) and on \(\ell\).

The tensor family \(k_z\) and its first normal derivative agree
across zero as smooth tensors on \(X\), by
\eqref{eq:glue-endpoint-jets}. The same is true after any number
of tangential derivatives. In particular, in each product chart
\(G\) is \(C^1\), its weak second derivatives are bounded,
and they equal the ordinary second derivatives on the two open
sides. There is no singular measure at the interface: the
coefficient of a possible Dirac term in a second derivative
would be a jump of a first derivative, and every such jump is
zero.

Choose \(\rho\in C_c^\infty((-1,1))\) with
\(\rho\ge0\) and \(\int\rho=1\). For \(0<\nu<\ell/2\),
put \(\rho_\nu(t)=\nu^{-1}\rho(t/\nu)\) and define
\begin{equation}\label{eq:glue-tensor-mollifier}
 k_z^\nu(x)=\int\rho_\nu(z-s)k_s(x)\,ds,
 \qquad G^\nu=dz^2+k_z^\nu
 \quad (|z|\le2\ell).
\end{equation}
Only \(|s|<3\ell\) contributes to this integral, so no extension
outside the chosen collar is needed. The integration takes place
in the single vector space
\(\operatorname{Sym}^2T_x^*X\), for fixed \(x\). In
particular the same nonnegative kernel is used for every tensor
component. Product coordinate changes on \(X\) are independent
of \(z\), so this operation is intrinsically defined on the
collar; no coordinate-dependent patching of separately smoothed
components is needed. Positive definiteness is preserved by the
integral. Tangential differentiation commutes with it, and
normal differentiation acts on the smooth kernel. Thus
\(G^\nu\) is smooth in all variables. Uniform continuity of
the first derivatives of \(G\) on the compact collar gives
\begin{equation}\label{eq:glue-C1-convergence}
 \|G^\nu-G\|_{C^1(X\times[-2\ell,2\ell])}\longrightarrow0.
\end{equation}
To control curvature, we estimate the linear second-derivative
terms and the nonlinear first-derivative terms separately.

\subsection{Curvature under the tensor-valued smoothing}
In fixed product coordinates, set
\[
 \Gamma_{ABC}(G)
 =\tfrac12(G_{BC,A}+G_{AC,B}-G_{AB,C}).
\]
With all indices lowered, the curvature formula for a general
positive definite metric is
\begin{align}
 R_{ABCD}(G)={}&\tfrac12\bigl(
 G_{BD,AC}-G_{BC,AD}-G_{AD,BC}+G_{AC,BD}\bigr)\notag\\
 &+G^{EF}\bigl(
 \Gamma_{ACE}(G)\Gamma_{BDF}(G)
 -\Gamma_{BCE}(G)\Gamma_{ADF}(G)\bigr).
 \label{eq:glue-coordinate-curvature}
\end{align}
In particular,
\begin{equation}\label{eq:glue-linear-nonlinear}
 R(G)=\mathcal L(D^2G)+\mathcal Q(G,DG),
\end{equation}
where \(D\) here denotes coordinate differentiation,
\(\mathcal L\) is a constant-coefficient linear map of second
coordinate derivatives, and \(\mathcal Q\) is the smooth
expression in the second line of
\eqref{eq:glue-coordinate-curvature}. The latter involves no
second derivative. Its dependence on \(G\) is smooth on the
cone of positive definite matrices.

Apply \eqref{eq:glue-linear-nonlinear} to \(G^\nu\).
Convolution commutes with all coordinate derivatives in the
weak sense. The absence of a singular measure in \(D^2G\)
therefore gives the exact identity
\begin{equation}\label{eq:glue-curvature-error}
 R(G^\nu)=\rho_\nu*R(G)+E_\nu,
 \qquad
 E_\nu=\mathcal Q(G^\nu,DG^\nu)
                  -\rho_\nu*\mathcal Q(G,DG).
\end{equation}
The curvature of \(G\) on the right is evaluated on its two
smooth sides; its values on the interface, a set of measure
zero in the convolution variable, are immaterial.

To quantify the error, work on a smaller fixed coordinate
chart whose closure lies in a product chart. Uniform positive
definiteness and \eqref{eq:glue-C1-convergence} put the pairs
\((G^\nu,DG^\nu)\) and \((G,DG)\) in a fixed compact subset
of the domain of \(\mathcal Q\). The mean value estimate and
a modulus of continuity \(\omega_{\mathcal Q}\) of the
continuous function \(\mathcal Q(G,DG)\) yield
\begin{equation}\label{eq:glue-error-control}
 \|E_\nu\|_{C^0}
 \le C\|G^\nu-G\|_{C^1}+\omega_{\mathcal Q}(\nu)
 \longrightarrow0.
\end{equation}
The finitely many chart estimates give uniform convergence on
the collar. Although the separate expressions in
\eqref{eq:glue-linear-nonlinear} are coordinate expressions,
the difference \(E_\nu\) in
\eqref{eq:glue-curvature-error} is a tensor: all curvature
tensors in the convolution are identified using the product
bundle \(TX\oplus\Rr\partial_z\), and the chart changes on
this bundle are independent of \(z\).

We now pass from \eqref{eq:glue-error-control} to a lower bound
on every two-plane. Let
\(\overline G=dz^2+h\), using the common boundary metric
transported by the product structure. On the fixed strip choose
constants \(0<m_*\le M_*<\infty\) such that
\begin{equation}\label{eq:glue-metric-equivalence}
 m_*\overline G\le G\le M_*\overline G.
\end{equation}
The same bounds hold for \(G^\nu\) by
\eqref{eq:glue-tensor-mollifier}. For a decomposable bivector
\(\xi\in\Lambda^2(T_xX\oplus\Rr\partial_z)\), fixed as
\(z\) varies, the positive sectional-curvature bound on each
smooth side gives
\begin{equation}\label{eq:glue-ae-positive}
 \mathcal R_{G(x,z)}(\xi,\xi)
 \ge c_*|\xi|_{G(x,z)}^2
 \ge c_*m_*^2|\xi|_{\overline G}^2
 \quad\text{for almost every }z.
\end{equation}
Here \(\mathcal R\) is the bilinear form on \(\Lambda^2\)
fixed in Section~\ref{subsec:conventions}; in particular
\(\mathcal R_G(A\wedge B,A\wedge B)=G(R(A,B)B,A)\).
The factor \(m_*^2\) follows by taking the second exterior
power of the metric comparison in
\eqref{eq:glue-metric-equivalence}.

Convolve \eqref{eq:glue-ae-positive} with the common nonnegative
kernel. The product identification preserves decomposability,
so the same bivector can be used in every term of the integral.
The linear conversion of the all-lowered tensor \(R\) to
\(\mathcal R\) in \eqref{eq:glue-curvature-error}, and
\eqref{eq:glue-error-control}, then give
\[
 \mathcal R_{G^\nu}(\xi,\xi)
 \ge (c_*m_*^2-e_\nu)|\xi|_{\overline G}^2,
 \qquad e_\nu\longrightarrow0.
\]
Here \(e_\nu\) is an upper bound for the operator norm of the
induced error bilinear form on \(\Lambda^2\), measured with
\(\overline G\). Choose \(\nu\) so small that
\(e_\nu\le c_*m_*^2/2\). Since
\(|\xi|_{G^\nu}^2\le M_*^2|\xi|_{\overline G}^2\), we obtain
\begin{equation}\label{eq:glue-smoothed-positive}
 \Ksec_{G^\nu}\ge\frac{c_*m_*^2}{2M_*^2}>0
 \qquad (|z|\le2\ell).
\end{equation}
The argument uses only decomposable bivectors, with the
nonlinear curvature error \(E_\nu\) controlled by
\eqref{eq:glue-error-control}.

\subsection{Localization and completion of the proof}
To leave the old metric unchanged away from the interface,
choose \(\zeta\in C_c^\infty((-2\ell,2\ell))\) with
\(0\le\zeta\le1\) and \(\zeta=1\) on
\([-\ell,\ell]\). Define
\begin{equation}\label{eq:glue-localization}
 \widehat k_z=k_z+\zeta(z)(k_z^\nu-k_z),\qquad
 \widehat G=dz^2+\widehat k_z.
\end{equation}
This is smooth at the interface because it equals \(G^\nu\)
there, and is smooth elsewhere because \(G\) is smooth away
from the interface. It is positive definite, being a convex
combination of the two positive definite slice metrics. Its
curvature is positive where \(\zeta=1\) by
\eqref{eq:glue-smoothed-positive}, and where \(\zeta=0\) it
is unchanged.

It remains to check the transition region
\(\ell\le|z|\le2\ell\). For \(\nu<\ell/2\), the
convolution at these points samples only one smooth side of
\(G\). On this compact region, convolution therefore gives
\(k^\nu-k\to0\) in \(C^2\), including all tangential and
normal derivatives of total order at most two. Since \(\zeta\)
is fixed,
\[
 \|\widehat G-G\|_{C^2(\{\ell\le|z|\le2\ell\})}
 \le C_\zeta
       \|k^\nu-k\|_{C^2(\{\ell\le|z|\le2\ell\})}
 \longrightarrow0.
\]
The coordinate curvature formula and uniform positive
definiteness imply convergence of sectional curvatures,
uniformly on the compact bundle of two-planes in this region.
The old metric has the positive lower bound \(c_*\), so after
decreasing \(\nu\) again the localized metric is also strictly
positively curved there. Thus the localized modification
preserves positive sectional curvature by \(C^2\) closeness
on this smooth transition region.

Carry out this construction in the two disjoint strips about
\(z=-\tau\) and \(z=\tau\), choosing the smoothing widths
small enough for both. The metric is then smooth and strictly
positively curved everywhere in the collar, and equals the
original metrics outside the chosen neighborhood. This proves Lemma~\ref{lem:gluing}.


\begin{thebibliography}{99}
\bibitem{BS}
S.~Brendle and R.~Schoen,
\emph{Manifolds with \(1/4\)-pinched curvature are space forms},
J. Amer. Math. Soc. \textbf{22} (2009), no.~1, 287--307.
\href{https://doi.org/10.1090/S0894-0347-08-00613-9}{doi:10.1090/S0894-0347-08-00613-9}.

\bibitem{DMR}
C.~Dur\'an, A.~Mendoza, and A.~Rigas,
\emph{Blakers--Massey elements and exotic diffeomorphisms of
\(S^6\) and \(S^{14}\) via geodesics},
Trans. Amer. Math. Soc. \textbf{356} (2004), no.~12, 5025--5043.
\href{https://doi.org/10.1090/S0002-9947-04-03469-5}{doi:10.1090/S0002-9947-04-03469-5}.

\bibitem{DPR}
C.~Dur\'an, T.~P\"uttmann, and A.~Rigas,
\emph{An infinite family of Gromoll--Meyer spheres},
Arch. Math. (Basel) \textbf{95} (2010), no.~3, 269--282.
\href{https://doi.org/10.1007/s00013-010-0161-x}{doi:10.1007/s00013-010-0161-x}.

\bibitem{EK}
J.-H.~Eschenburg and M.~Kerin,
\emph{Almost positive curvature on the Gromoll--Meyer sphere},
Proc. Amer. Math. Soc. \textbf{136} (2008), no.~9, 3263--3270.
\href{https://doi.org/10.1090/S0002-9939-08-09429-X}{doi:10.1090/S0002-9939-08-09429-X}.

\bibitem{GKS}
S.~Goette, M.~Kerin, and K.~Shankar,
\emph{Highly connected \(7\)-manifolds and non-negative sectional curvature},
Ann. of Math. (2) \textbf{191} (2020), no.~3, 829--892.
\href{https://doi.org/10.4007/annals.2020.191.3.3}{doi:10.4007/annals.2020.191.3.3}.

\bibitem{GM}
D.~Gromoll and W.~Meyer,
\emph{An exotic sphere with nonnegative sectional curvature},
Ann. of Math. (2) \textbf{100} (1974), no.~2, 401--406.
\href{https://doi.org/10.2307/1971078}{doi:10.2307/1971078}.

\bibitem{GZ}
K.~Grove and W.~Ziller,
\emph{Curvature and symmetry of Milnor spheres},
Ann. of Math. (2) \textbf{152} (2000), no.~1, 331--367.
\href{https://doi.org/10.2307/2661385}{doi:10.2307/2661385}.

\bibitem{GFL}
S.~Guo, E.~X.~Fang, and J.~Lu,
\emph{A two-stage construction of positive curvature on the Gromoll--Meyer sphere},
preprint, \href{https://arxiv.org/abs/2609.12882v1}{arXiv:2609.12882v1}
(2026).

\bibitem{KM}
M.~A.~Kervaire and J.~Milnor,
\emph{Groups of homotopy spheres. I},
Ann. of Math. (2) \textbf{77} (1963), no.~3, 504--537.
\href{https://doi.org/10.2307/1970128}{doi:10.2307/1970128}.

\bibitem{Milnor}
J.~Milnor,
\emph{On manifolds homeomorphic to the \(7\)-sphere},
Ann. of Math. (2) \textbf{64} (1956), no.~2, 399--405.
\href{https://doi.org/10.2307/1969983}{doi:10.2307/1969983}.

\bibitem{ON}
B.~O'Neill,
\emph{The fundamental equations of a submersion},
Michigan Math. J. \textbf{13} (1966), no.~4, 459--469.
\href{https://doi.org/10.1307/mmj/1028999604}{doi:10.1307/mmj/1028999604}.

\bibitem{Ouyang}
Z.~Ouyang,
\emph{A positively curved metric on the Gromoll--Meyer sphere},
preprint, \href{https://arxiv.org/abs/2609.11484v1}{arXiv:2609.11484v1}
(2026).

\bibitem{PW}
P.~Petersen and F.~Wilhelm,
\emph{An exotic sphere with positive sectional curvature},
preprint, \href{https://arxiv.org/abs/0805.0812v3}{arXiv:0805.0812v3}
(2008).

\bibitem{RW}
P.~Reiser and D.~J.~Wraith,
\emph{A generalization of the Perelman gluing theorem and applications},
preprint, \href{https://arxiv.org/abs/2308.06996v2}{arXiv:2308.06996v2}
(2024).

\bibitem{sagemath}
The~Sage~Developers.
\newblock {\em {S}ageMath, the {S}age {M}athematics {S}oftware {S}ystem (Version 10.9)},
\newblock 2026.
\newblock \href{https://sagemath.org}{https://sagemath.org}.



\bibitem{Wilhelm}
F.~Wilhelm,
\emph{An exotic sphere with positive curvature almost everywhere},
J. Geom. Anal. \textbf{11} (2001), no.~3, 519--560.
\href{https://doi.org/10.1007/BF02922018}{doi:10.1007/BF02922018}.


\end{thebibliography}
\end{document}